\documentclass[12pt,leqno]{article}

\usepackage[lite]{amsrefs}
\usepackage{latexsym,enumerate,pstricks}

\usepackage[notref, notcite]{}
\usepackage{amssymb, amsfonts, eucal,amsmath,amsthm}

\newtheorem{theorema}{Theorem}

\newtheorem{theorem}{Theorem}[section]

\newtheorem{lemma}[theorem]{Lemma}

\newtheorem{remark}[theorem]{Remark}
\newtheorem{definition}[theorem]{Definition}

\newtheorem{corollary}[theorem]{Corollary}

\newcommand{\sect}[1]{\section{#1} \setcounter{equation}{0} }

\newcommand{\norm}[2]{\left\|#1\right\|_{#2}}

\newcommand{\dlt}{\rho}
\newcommand{\dn}{\rho_n}
\newcommand{\dm}{\rho_m}
\newcommand{\dnx}{\dn(x)}
\newcommand{\dmx}{\dm(x)}

\newcommand{\gp}{\lambda_p}

\newcommand{\ds}{\displaystyle}

 \newcommand{\ec}{\end{comment}}
\newcommand{\bc}{ \begin{comment}
 }
 \newcommand{\bsmall}{\begin{scriptsize} \vspace{1cm} \mbox{}\\
  $\downarrow\rule{\textwidth}{0.1mm}  \downarrow$
 \setlength{\baselineskip}{0.3mm}
 }
 \newcommand{\esmall}{\noindent $\uparrow \rule{\textwidth}{0.1mm}\uparrow$ \end{scriptsize}\vspace{1cm} }

\newcommand{\E}{{\mathcal E}}

\newcommand{\I}{J}

\newcommand{\ccc}{{\mathfrak  \vartheta}}

\newcommand{\A}{{\mathcal A}}
\newcommand{\B}{{\mathcal B}}
\newcommand{\CC}{{\mathcal C}}
\newcommand{\D}{{\mathcal D}}

\newcommand{\che}{{\mathfrak t}}

\newcommand{\andd}{\quad\mbox{\rm and}\quad}

\newcommand\e{{\varepsilon}}

\newcommand\w{{\omega}}
\newcommand{\Q}{{\mathcal Q}}

\def\be  {\begin{equation}}
\def\ee  {\end{equation}}
\def\ba  {\begin{eqnarray}}
\def\ea  {\end{eqnarray}}
\def\baa {\begin{eqnarray*}}
\def\eaa {\end{eqnarray*}}
\newenvironment{comment}[2]
{\bgroup\vspace{7pt}
\begin{tabular}{|p{5in}|}
\hline \qquad \bf \footnotesize Comment -- to be deleted in the final version \\
\hline
\quad\sl\footnotesize #1#2} {\\ \hline \end{tabular}
\vspace{7pt}\indent\egroup}

\def\updots{\mathinner{\mkern
1mu\raise 1pt \hbox{.}\mkern 2mu \mkern 2mu \raise
4pt\hbox{.}\mkern 1mu \raise 7pt\vbox {\kern 7 pt\hbox{.}}} }

\def \esssup{\mathop{\rm ess\: sup}\nolimits}

  \newcommand{\K}{{\mathit{K}}}
   \newcommand{\Rea}{{\mathit{R}}}

\newcommand{\C}{\mathbb C}

\newcommand{\W}{{\mathcal{W}}}
\newcommand{\Wdg}{\W^{\delta, \gamma}}

\newcommand{\R}{\mathbb R}

\newcommand{\N}{\mathbb N}

 \renewcommand{\a}{\alpha}
\renewcommand{\b}{\beta}
\newcommand{\ineq}[1]{(\ref{#1})}

\newcommand{\ie}{{\em i.e., }}
\newcommand{\eg}{{\em e.g. }}

\newcommand{\bpic}{
\begin{center}
}

\newcommand{\epic}{
\endpspicture
\end{center}
}

\newcommand{\st}{\;\; \big| \;\;}
\renewcommand{\L}{\mathbb{L}}
\newcommand{\Lp}{\L_p}

\newcommand{\Poly}{\Pi}
 \newcommand{\AC}{\mathrm{AC}}
  \newcommand{\loc}{\mathrm{loc}}

\newcommand{\Dom}{{\mathfrak{D}}}

\newcommand{\thm}[1]{Theorem~\ref{#1}}
\newcommand{\lem}[1]{Lemma~\ref{#1}}
\newcommand{\cor}[1]{Corollary~\ref{#1}}

\title{{Polynomial approximation with doubling weights}
}

\author{
Kirill A.  Kopotun\thanks{Department of Mathematics, University of
Manitoba, Winnipeg, Manitoba, R3T 2N2, Canada ({\tt
kopotunk@cc.umanitoba.ca}). Supported by NSERC of Canada.}    }

\begin{document}

\maketitle

\begin{abstract}
A nonnegative   function $w\in\L_1[-1,1]$ is called a doubling weight if there is a constant $L$   such that $w(2I) \leq L w(I)$,
for all intervals $I \subset [-1,1]$, where $2I$ denotes the interval having the same center as $I$ and twice as large as $I$, and $w(I) := \int_I w(u) du$.
In this paper, we establish direct   and inverse   results for weighted approximation by algebraic polynomials in the $\Lp$, $0<p\leq \infty$, (quasi)norm weighted by
$w_n := \dnx^{-1} \int_{x-\dnx}^{x+\dnx} w(u) du$, where $\dnx := n^{-1}\sqrt{1-x^2} + n^{-2}$ and $w$ is a doubling weight.

Among other things, we prove that, for a doubling weight $w$, $0< p\leq\infty$, $r\in\N_0$, and $0<\alpha <r+1 - 1/\gp$, we have
\[  \tag{$\ast$}
E_n(f)_{p, w_n} = O(n^{-\a}) \iff  \w_\varphi^{r+1}(f, n^{-1})_{p, w_n} = O(n^{-\a}) ,
\]
where $\gp := p$ if $0 < p < \infty$,  $\gp :=1$ if $p=\infty$,
$\norm{f}{p,w} := \left( \int_{-1}^1 |f(u)|^p w(u) du \right)^{1/p}$,
$\norm{f}{\infty,w} := \esssup_{u\in [-1,1]} \left(|f(u)| w(u)\right)$,
$\w_\varphi^r(f, t)_{p, w} :=  \sup_{0<h\leq t}  \norm{ \Delta_{h\varphi(\cdot)}^r(f,\cdot)}{p, w}$,
$E_n(f)_{p, w} :=   \inf_{P_n\in\Poly_n} \norm{f-P_n}{p,w}$, and
$\Poly_n$ is the set of all algebraic polynomials of degree $\leq n-1$.

We will also introduce classes of doubling weights $\Wdg$ with parameters  $\delta, \gamma \geq 0$ that are used to describe the behavior of $w_n(x)/w_m(x)$ for $m\leq n$. It turns out that
every class $\Wdg$ with $(\delta, \gamma) \in \Upsilon := \left\{(\delta, \gamma)\in\R^2 \st \delta\geq 1, \gamma \geq 0, \delta + \gamma \geq 2 \right\}$ contains all doubling weights $w$,
and for each pair $(\delta, \gamma) \not\in \Upsilon$, there is a doubling weight not in $\Wdg$.
We will establish inverse theorems and equivalence results similar to  ($\ast$)   for doubling weights from classes $\Wdg$. Using the fact that $1\in \W^{0,0}$, we get the well known inverse results and equivalences of type
($\ast$) for unweighted polynomial approximation as an immediate corollary.

Equivalence type results involving related $\K$-functionals and realization type results (obtained as  corollaries of our estimates) are also discussed.

Finally, we mention that ($\ast$)   closes a gap left in the paper by G. Mastroianni and V. Totik ``Best Approximation and moduli of smoothness for doubling weights'', J. Approx. Theory {\bf 110} (2001), 180-199, where ($\ast$) was established for $p=\infty$ and $\w_\varphi^{r+2}$ instead of $\w_\varphi^{r+1}$ (it was shown there that, in general, ($\ast$) is not valid for $p=\infty$ if $\w_\varphi^{r+1}$ is replaced by $\w_\varphi^{r}$).

\end{abstract}

\sect{Introduction and main results}

As usual,   $\Lp(I)$, $0< p \leq \infty$, denotes the set of all measurable on $I$ functions $f$ equipped with the (quasi)norm  $\norm{f}{\Lp(I)}$, 
and $\norm{f}{p} := \norm{f}{\Lp[-1,1]}$.
For a (nonnegative) weight function $w$
and   $I\subset [-1,1]$ denote
\[
\norm{f}{\Lp(I),w} := \left( \int_I |f(u)|^p w(u) du \right)^{1/p} \andd
\norm{f}{\L_\infty(I),w} := \esssup_{u\in I} \left(|f(u)| w(u)\right)   .
\]
We also let $\norm{f}{p,w} := \norm{f}{\Lp[-1,1],w}$.
Note that $\norm{f}{\Lp(I),w} = \norm{w^{1/p} f}{\Lp(I)}$  if $0< p <\infty$, and $\norm{f}{\L_\infty(I),w}= \norm{w  f}{\L_\infty(I)}$ if $p=\infty$.
%
%
Assuming for convenience that $w$ is identically zero outside $[-1,1]$, we recall that   $w$ is called a doubling weight if there is a constant $L$ (the so-called doubling constant of the weight $w$) such that
\[
\int_{2I} w(u) du \leq L \int_{I} w(u) du ,
\]
for all intervals $I \subset [-1,1]$, where $2I$ denotes the interval having the same center as $I$ and twice as large as $I$.
We note that the class of doubling weights is quite large, for example, all generalized Jacobi weights are doubling. Also, they are closely related to   Muckenhoupt's $A_p$, $1\leq p<\infty$, weights all of which are contained in the so-called   $A_\infty$ class of weights that assign to a subset of an interval $I\subset[-1,1]$ a ``fair'' share of the weight of $I$. We refer the reader to \cite[Chapter V]{stein} for details on $A_p$ and $A_\infty$ classes, their characterizations and properties, and to the  series of papers \cites{mt1998,mt1999,mt2000,mt2001} for detailed discussions of various properties of doubling weights.

Following \cites{mt1998,mt1999,mt2000,mt2001}, for a weight $w\in\L_1$,  we set
\[
w_n(x) := {1 \over \dnx} \int_{x-\dnx}^{x+\dnx} w(u) du ,
\]
where $\dnx := n^{-1}\sqrt{1-x^2} + n^{-2}$. 

If $r\in\N$, the weighted modulus of smoothness is defined by
\[
\w_\varphi^r(f, t)_{p, w} :=  \sup_{0<h\leq t}  \norm{ \Delta_{h\varphi(\cdot)}^r(f,\cdot)}{p, w}  ,
\]
where $\varphi(x) := \sqrt{1-x^2}$ and
\[\Delta_h^r(f,x, [a,b]):=\left\{
\begin{array}{ll} \ds
\sum_{i=0}^r  {r \choose i}
(-1)^{r-i} f(x-rh/2+ih),&\mbox{\rm if }\, x\pm rh/2  \in [a,b] \,,\\
0,&\mbox{\rm otherwise},
\end{array}\right.\]
is  the $r$th symmetric difference, $\Delta_h^r(f,x) := \Delta_h^r(f,x, [-1,1])$.

Throughout this paper, we use the standard notation, \ie $\N$ is the set of all positive integers, $\N_0 := \N \cup\{0\}$, $\R_+ := [0,\infty)$, $\Poly_n$ is the set of all algebraic polynomials of degree $\leq n-1$, $c$ are  positive constants that may be different even if they occur in the same line, $A\sim B$ means that $cA \leq B \leq cA$, for some constants $c$ that do not depend on the ``important'' variables (what's ``important'' is usually clear from the context).
We also use the notation $c_*$, $c^*$ and $c_i$ ($i\in\N_0$)  for constants that we need to refer to, but those stay fixed only inside the lemmas where they are introduced (to make this explicit, we use $c_*$, for example, in several statements, but none of these constants are  assumed to be the same).
Additionally,
$E_{n}(f)_{p,w} := \inf_{P_n\in\Poly_n} \norm{f-P_n}{p,w}$ is the rate of best weighted approximation with weight $w$ of $f$ by algebraic polynomials of degree $\leq n-1$.

The following theorem that motivated   this work was proved in \cite[Theorem 1.1]{mt2001}.

\begin{theorema} \label{thmA}
Let $w$ be a doubling weight and $r\in\N$. Then there is a constant $c_*$ depending only on $r$ and the doubling constant of $w$ such that we have for any $f$
\[
E_n(f)_{\infty, w_n} \leq c_* \w_\varphi^r(f, 1/n)_{\infty, w_n}.
\]
Conversely,
\[
\w_\varphi^{r+2} (f, n^{-1})_{\infty, w_n}  \leq  {c_* \over n^{r}}   \sum_{k=1}^{n}    k^{r-1}     E_{k}(f)_{\infty,w_{k}} .
\]
\end{theorema}
It immediately follows from \thm{thmA} that, for $0<\a<r$,
\be   \label{mtim}
E_n(f)_{\infty, w_n} = O(n^{-\a}) \iff  \w_\varphi^{r+2}(f, n^{-1})_{\infty, w_n} = O(n^{-\a}) ,
\ee
and it was shown in \cite[p. 183]{mt2001} that \ineq{mtim} is no longer true if $\w_\varphi^{r+2}$ is replaced by $\w_\varphi^{r}$ (in the case $r=1$).
Among other things, we show in this paper that \ineq{mtim} holds if $\w_\varphi^{r+2}$ is replaced by $\w_\varphi^{r+1}$ (see \cor{cor5.6}).

In this paper, we prove direct and inverse theorems for all $0<p\leq \infty$.
For example, we prove that, if  $w$ is a doubling weight,  $r\in\N$,  $0<p\leq \infty$ and $f\in\Lp[-1,1]$, then for every $n \geq r$,
\[
E_n(f)_{p, w_n}   \leq   c \w_\varphi^r(f, 1/n)_{p, w_n} .
\]
Conversely,
\[
\w_\varphi^r (f, n^{-1})_{p, w_n}  \leq  {c \over n^{r-1/\gp}}   \sum_{k=1}^{n}    k^{r-1 -1/\gp }     E_{k}(f)_{p,w_{k}} , \quad 1\leq p \leq \infty ,
\]
and
\[
\w_\varphi^r (f, \ccc n^{-1})_{p, w_n}  \leq  {c \over n^{r-1/p}}   \left( \sum_{k=1}^{n}    k^{rp-2}     E_{k}(f)_{p,w_{k}}^p \right)^{1/p}, \quad 0<p<1 ,
\]
where $0<\ccc\leq 1$ is some constant, $\gp := p$, if $ p <\infty$, and $\lambda_\infty := 1$, if $p=\infty$.

This implies that, for any doubling weight $w$, $r\in\N$,   $0< p \leq \infty$ and $f\in\Lp[-1,1]$, for $0<\a< r-1/\gp$ we have
\[
E_n(f, [-1,1])_{p, w_n} = O(n^{-\a}) \iff  \w_\varphi^r(f, n^{-1})_{p, w_n} = O(n^{-\a}) .
\]

 In fact, we prove the inverse theorems in a more general way. In Section~\ref{properties},
we   introduce classes of doubling weights $\Wdg$ with parameters  $\delta, \gamma \geq 0$ that are used to describe the behavior of $w_n(x)/w_m(x)$ for $m\leq n$. It turns out that
every class $\Wdg$ with $(\delta, \gamma) \in \Upsilon  = \left\{(\delta, \gamma)\in\R^2 \st \delta\geq 1, \gamma \geq 0, \delta + \gamma \geq 2 \right\}$ contains all doubling weights $w$, and for each class
$\Wdg$ with $(\delta, \gamma) \not\in \Upsilon$, there is a doubling weight not in this class.
 We will establish inverse theorems   for doubling weights from classes $\Wdg$ for all $\delta, \gamma\geq 0$. Since positive constants are doubling weights from the class  $\W^{0,0}$,
 we get the well known inverse results  for unweighted polynomial approximation as an immediate corollary.

The paper is organized as follows. In Section~\ref{properties}, we discuss several properties of weights $w_n$ and introduce classes $\Wdg$. An auxiliary result on a polynomial partition of unity that is crucial in our proof of direct results is introduced in Section~\ref{partition}. In Section~\ref{polyapprox}, we approximate the weights $w_n^{1/p}$ by polynomials from $\Poly_n$. Section~\ref{jacksonsection} is devoted to proving Jackson type (\ie direct) results on polynomial approximation with weights $w_n$. Markov-Bernstein type results are discussed in Section~\ref{mbsection}. A major part of this section is devoted to the case $0<p<1$ in preparation for inverse results for these $p$. The inverse theorems are proved in Section~\ref{inversesection} and, in Section~\ref{kfn}, we discuss some results on the equivalence of the moduli $\w_\varphi^r$ as well as the averaged moduli $\tilde \w_\varphi^r$ and  several $\K$-functionals and Realization functionals with weights $w_n$.

It is also worth mentioning that it seems possible to get the Jackson-type results in the case $1\leq p\leq \infty$ using the Jackson-Favard type inequalities proved in \cites{mt1998, mt1999} and equivalence of the moduli $\w_\varphi^r$ with weights $w_n$ and related $\K$-functionals that can be obtained following proofs in \cite{dt} (as was done in \cite{mt2001}, see also \cite{d20}). However, we opted for a different approach in this paper that works in the case $0<p<1$ as well.

\sect{Doubling weights and their subclasses} \label{properties}

In this section, we discuss several properties of $w_n$ that will be used in this paper.
First, we note that it was proved in \cite[Lemma 7.1]{mt2000} that     the doubling condition is equivalent to
\be \label{ineq1.1}
w_n(x) \leq K \left( 1 + n |x-y| + n \left| \varphi(x)- \varphi(y) \right| \right)^s w_n(y), \quad n\in\N, \;\; x,y\in [-1,1],
\ee
with some positive constants $K$ and $s$.

It immediately follows from \ineq{ineq1.1} (see also \cite[(2.3)]{mt2001}) that
\be \label{ineq1.2}
|x-y|\leq M \dnx \quad  \Longrightarrow   \quad    w_n(x) \sim w_n(y)         
\ee
with equivalence constants  depending only on $M$ and the doubling constant of $w$.

We will now discuss the relations between $w_n$ and $w_m$ for different $n$ and $m$.
First of all, it is evident that,   for any $x\in[-1,1]$,
\[
\dnx  \leq \dmx  \leq (n/m)^2 \dnx  \quad \mbox{\rm if} \quad    m \leq n .
\]
 Therefore, since $w$ is nonnegative we have
\be \label{msq}
w_n(x) \leq (n/m)^2 w_m (x) , \quad m\leq n .
\ee
Also, taking into account that $w$ is doubling and using \cite[Lemmas 7.1 and 2.1(vi)]{mt2000}     we have,  for $m\leq n$ and $M:=n/m$,
\[
w_m(x) \leq \dnx^{-1} \int_{x-M^2 \dnx}^{x+M^2\dnx}  w(u) du  \leq c w_n(x) ,
\]
for some constant $c$ that depends on $M$ and the doubling constant of $w$.
Hence, in particular,
\[
w_n(x) \sim w_m(x) , \quad \mbox{\rm if} \;\; n\sim m .
\]

It is rather obvious that \ineq{msq} cannot be improved uniformly for all $x\in [-1,1]$ (for all doubling weights $w$) in the sense that it is no longer valid if $(n/m)^2$ is replaced by $(n/m)^{2-\e}$, for any $\e>0$
(see also \lem{lemcrucial} below).
At the same time, it is clear that $(n/m)^2$ in \ineq{msq} can be replaced by $(n/m)$ for $x$ that are ``far'' from the endpoints of $[-1,1]$.
The following simple lemma makes this observation more precise and turns out to be crucial in our proofs of the inverse theorems.

\begin{lemma} \label{lemcrucial}
Let $w$ be a doubling weight, $m,n\in\N$ be such that $m\leq n$, and let 
\be \label{upsilon}
(\delta, \gamma) \in \Upsilon := \left\{(\delta, \gamma)\in\R^2 \st \delta\geq 1, \gamma \geq 0, \delta + \gamma \geq 2 \right\} .
\ee
Then
\be \label{crucial}
{w_n(x) \over w_m(x)} \leq
\left( {n \over m}\right)^\delta \left( 1 + {1 \over m\varphi(x)} \right)^\gamma , \quad -1\leq x\leq 1.
\ee
Moreover,  if $(\delta, \gamma) \not\in \Upsilon$ then there are doubling weights $w$ for which
\ineq{crucial} is not valid.
\end{lemma}

\begin{proof} Clearly, \ineq{crucial} is satisfied if $m=n$, and so we assume that $m\leq n-1$.
Since $\dnx \leq \dmx$, we conclude that
 \[
 \int_{x-\dnx}^{x+\dnx} w(u) du  \leq  \int_{x-\dmx}^{x+\dmx} w(u) du ,
 \]
and hence
\[
{w_n(x) \over w_m(x)} \leq {\dmx \over \dnx} = {n \over m} \cdot {\varphi(x) +1/m \over \varphi(x) +1/n} .
\]
Therefore,  \ineq{crucial} will be proved if we   show that
 \[
L:=  {\varphi(x) +1/m \over \varphi(x) +1/n} \leq \left( {n \over m}\right)^{\delta-1} \left( 1 + {1 \over m\varphi(x)} \right)^\gamma  =: R.
\]
 Now, if $\varphi(x) \geq (n-m)^{-1}$, then
 \begin{eqnarray*}
 L &\leq& 1 + {1 \over m\varphi(x)} = \left( 1 + {1 \over m\varphi(x)} \right)^\gamma \left( 1 + {1 \over m\varphi(x)} \right)^{1-\gamma} \\
 &\leq& \left( 1 + {1 \over m\varphi(x)} \right)^\gamma \max\left\{ 1, (n/m)^{1-\gamma} \right\} \leq R .
 \end{eqnarray*}
 If  $\varphi(x) <  (n-m)^{-1}$, then
 \[
L = 1 + {1/m-1/n \over  \varphi(x) +1/n} \leq \frac{n}{m} \leq    \left( {n \over m}\right)^{\delta+\gamma-1}       \leq R .
 \]
We will now construct examples showing that \ineq{crucial} is no longer valid if  $(\delta, \gamma) \not\in \Upsilon$.

Let $R_{\delta, \gamma}(x)$ denote the right-hand side of  \ineq{crucial}. Since $\lim_{x \to 1}  R_{\delta, \gamma}(x) = 0$ if $\gamma < 0$, it is obvious that $\gamma$ has to be nonnegative for \ineq{crucial} to hold (for example, if $\gamma <0$, then $1 \not\in \Wdg$  for any $\delta$).

Now, let  $w^c(x) = |x-c|^{-\a}$ with $c\in [-1,1]$ and $0<\a <1$. It is not difficult to see that $w^c$ is doubling and $(w^c)_n(x) \sim \min\{|x-c|^{-\a}, \dn(c)^{-\a}\}$.
Hence, if $m\leq n$, then
\[
{(w^1)_n(1-n^{-2}) \over (w^1)_m(1-n^{-2})} \sim \left(\frac nm  \right)^{2\a}  \andd R_{\delta, \gamma}(1-n^{-2}) \sim \left(\frac nm  \right)^{\delta+\gamma}.
\]
Hence, if $\delta+\gamma < 2$, then \ineq{crucial} does not hold for the doubling weights $w^1$ with $\max\{(\delta+\gamma)/2, 0\} < \alpha <1$.
Also,
\[
{(w^0)_n(0) \over (w^0)_m(0)} \sim \left(\frac nm  \right)^{\a}  \andd R_{\delta, \gamma}(0) \sim \left(\frac nm  \right)^{\delta},
\]
and so, if $\delta < 1$, then  \ineq{crucial} does not hold for the doubling weights $w^0$ with $\max\{\delta, 0\}  < \alpha <1$.
\end{proof}

\begin{definition} \label{defcrucial}
Let $\delta,\gamma \geq 0$.
We say that a doubling weight $w$ belongs to the class $\Wdg_\Lambda$ if, for all $m,n\in\N$ such that  $m\leq n$ and all $x\in [-1,1]$,
\be \label{touse}
w_n(x) \varphi(x)^\gamma \leq \Lambda n^\delta m^{\gamma-\delta} \dm(x)^\gamma w_m(x) ,
\ee
for some constant   $\Lambda$ which may depend only on the weight $w$, parameters $\delta$ and $\gamma$, and is independent of $m$, $n$ and $x$.
We also denote
\[
\Wdg := \left\{ w \st w\in \Wdg_\Lambda \; \mbox{\rm for some } \Lambda >0 \right\}.
\]
\end{definition}

We remark that \ineq{touse} with $\Lambda =1$ is equivalent to \ineq{crucial} which is the reason for Definition~\ref{defcrucial}. Also, it is  evident that $\W^{\delta_1, \gamma_1} \subset \W^{\delta_2, \gamma_2}$ if $\delta_1 \leq \delta_2$ and $\gamma_1\leq \gamma_2$.

\begin{remark} \label{remcr}
 \lem{lemcrucial} implies that all doubling weights belong to the class $\Wdg_1$  if $(\delta, \gamma)\in \Upsilon$, where $\Upsilon$ is defined in \ineq{upsilon}. Moreover, for any pair $(\delta, \gamma)\not\in \Upsilon$, there is a doubling weight $w$ such that $w\not\in \Wdg$.
\end{remark}

Of course, there are many doubling weights belonging to the classes $\W^{\delta, \gamma}$ with $(\delta, \gamma)\not\in \Upsilon$. For example, any nonzero constant weight   belongs to $\W^{0,0}$.
The weight $w^*(x) = (1-x)^{-\a}$, $0<\a <1$,   belongs to $\Wdg$ for all $(\delta,\gamma)\in\R_+^2$ such that $\delta+\gamma \geq 2\a$.
The weight  $w_*(x) = |x|^{-\a}$, $0<\a <1$, belongs to $\Wdg$ for all $(\delta,\gamma)\in\R_+^2$ such that $\delta \geq \a$.
Hence, the doubling weight $W(x) := |x|^{-\a} (1-x)^{-\b}$, $0<\a,\b <1$,  which is a combination of $w^*$ and $w_*$ belongs to $\Wdg$ for all $(\delta,\gamma)\in\R_+^2$ such that $\delta \geq \a$ and $\delta+\gamma \geq 2\b$.

Following \cites{mt2000, mt2001} we say that a weight $w$ satisfies the $A^*$ property if there is a constant $c^*$ such that, for all $I\subset [-1,1]$ and $x\in I$,
\[
w(x) \leq {c^* \over |I|} \int_I w(u) du .
\]
Then $w$ is doubling and \ineq{ineq1.2} implies that $w_m(x)\sim w_m(u)$ if $|x-u| \leq M \dmx$. Hence, if $m\leq n$, taking into account that $\dnx \leq \dmx$ we conclude that $w_m(x)\sim w_m(u)$ for $|x-u| \leq  \dnx$. Therefore, denoting $J_m(u) :=  [u-\dm(u), u+\dm(u) \cap [-1,1]$, we have (see also \cite[p. 189]{mt2001})
\begin{eqnarray*}
w_n(x) &= & {1 \over \dnx} \int_{x-\dnx}^{x+\dnx} w(u) du  \\
& \leq &
{1 \over \dnx} \int_{x-\dnx}^{x+\dnx} \left(  { c^* \over \left| J_m(u)  \right|     } \int_{J_m(u)} w(v) dv   \right) du \\
&\leq& {1 \over \dnx} \int_{x-\dnx}^{x+\dnx} \left(  { c^* \over  \dm(u) } \int_{u-\dm(u)}^{u+\dm(u)} w(v) dv   \right) du \\
& = & { c^* \over   \dnx } \int_{x-\dnx}^{x+\dnx} w_m(u) du \sim { 1\over    \dnx } \int_{x-\dnx}^{x+\dnx} w_m(x) du  \sim  w_m(x) .
\end{eqnarray*}
Therefore,  we can make the following assertion.

\begin{remark} \label{remast}
Any weight $w$ that satisfies the $A^*$ property is in the class $\W^{0,0}_\Lambda$ with the constant $\Lambda$ that depends only on the doubling constant of $w$.
\end{remark}

\medskip

Finally, we will need the following technical lemma that will be quite useful in the proofs of  direct results
(note that $x_i$, $I_i$ and $\psi_i$ are defined at the beginning of Section~\ref{partition}).

\begin{lemma}
For a doubling weight $w$, $n\in\N$ and all  $1\leq i\leq n$, $x\in [-1,1]$ and $y\in I_i$, we have
\be \label{ineq1.3}
w_n(x) \leq c \psi_i(x)^{-s} w_n(y) \andd w_n(y) \leq c \psi_i(x)^{-s} w_n(x),
\ee
where constants $c$ and parameter $s\geq 0$ depend only on the doubling constant of $w$.
\end{lemma}

\begin{proof}
 Taking into account that, for $1\leq i\leq n-1$,  $|I_i| \sim \varphi(x_i)/n$   and  $\psi_i(x)^{-1} = 1 + |x-x_i|/|I_i| \sim 1 + n |x-x_i|/\varphi(x_i)$,
and using \ineq{ineq1.1} we have
\begin{eqnarray*}
\lefteqn{\max \left\{ w_n(x)/w_n(x_i), w_n(x_i)/w_n(x) \right\}   \leq  c \left( 1 + n |x-x_i| + n \left| \varphi(x)- \varphi(x_i) \right| \right)^s} \\
& \leq & c \left( 1 + n |x-x_i| +  { n|x^2-x_i^2| \over \varphi(x) + \varphi(x_i)}   \right)^s
 \leq  c
\left( 1 + n |x-x_i| +  { 2n|x -x_i | \over   \varphi(x_i)}   \right)^s \\
&\leq &
c \left( 1 +    { 3n|x -x_i | \over   \varphi(x_i)}   \right)^s
  \leq
c \psi_i(x)^{-s} , \quad 1\leq i\leq n-1 .
\end{eqnarray*}
Therefore, observing that \ineq{ineq1.2} implies that $w_n(u)\sim w_n(y)$, for $u,y\in I_i$, $1\leq i \leq n$, and $\psi_n(x) \sim \psi_{n-1}(x)$, we get \ineq{ineq1.3}.
\end{proof}

\sect{Partition of unity} \label{partition}

First, we recall the usual setup  for polynomial approximation (see \eg \cite{sh-book}).
Let $(x_i)_{i=0}^n$ be the Chebyshev partition of $[-1,1]$, \ie
 $x_i = \cos(i\pi/n)$, $0\leq i\leq n$, $I_i := [x_i, x_{i-1}]$, $1\leq i\leq n$,
 \[
 \psi_i :=    \psi_i(x) :=  \frac{|I_i|}{|x-x_i|+|I_i|} , \quad 1\leq i \leq n .
 \]
Then, for $1\leq i\leq n$,
 \[
 \che_i(x) := \left( \cos 2n \arccos x \over x - x_i^0 \right)^2 +
 \left( \sin 2n \arccos x \over x - \bar x_i  \right)^2
 \]
 is an algebraic polynomial of degree $4n-2$, where $\bar x_i :=  \cos(i\pi/n - \pi/2n)$, $1\leq i\leq n$,
$x_i^0 := \cos(i\pi/n -  \pi/4n)$, $1\leq i < n/2$, and  $x_i^0 := \cos(i\pi/n - 3\pi/4n)$, $n/2 \leq i \leq n$.
The following properties of the Chebyshev partition will often be used: 
\[
|I_i| \sim \dnx ,\; x\in I_i, \; 1\leq i\leq n , \andd |I_i| \sim |I_{i+1}| , \; 1\leq i \leq n-1 .
\]
It is also convenient to denote
\[
\chi_i(x) := \chi_{[x_i, 1]}(x) =
\begin{cases}
1, & \mbox{\rm if } x_i \leq x \leq 1, \\
0, & \mbox{\rm otherwise}.
\end{cases}
\]

The crucial (obvious) property of   polynomials $\che_i$, $1 \leq i \leq n$, is
\[
\min\left\{ (x - x_i^0)^{-2},  (x - \bar x_i)^{-2} \right\} \leq \che_i(x) \leq \max\left\{ (x - x_i^0)^{-2},  (x - \bar x_i)^{-2} \right\} ,
\]
which implies
\be \label{propti}
\che_i(x) \sim \left( |x-x_i|+ |I_i| \right)^{-2}
\ee
uniformly for $x\in [-1,1]$.

There exists an absolute (positive) constant  $c_*$  such that,
for $\mu, \e_1, \e_2 \in \N_0$ satisfying $\mu \geq c_* \max\{\e_1,\e_2, 1\}$,
\[
T_i(x) := T_i (n, \mu, \e_1, \e_2)(x) := \lambda_i \int_{-1}^x   (y-x_i)^{\e_1} (x_{i-1}-y)^{\e_2} \che_i^\mu(y)\, dy
\]
is a polynomial of degree $(4n-2)\mu + \e_1 +\e_2 +1$, where
\be \label{propt2}
\lambda_i := \left(\int_{-1}^1   (y-x_i)^{\e_1} (x_{i-1}-y)^{\e_2} \che_i^\mu(y)\, dy\right)^{-1} \sim |I_i|^{2\mu-\e_1-\e_2 -1} ,
\ee
and so $T_i(1)=1$ (see \eg \cite[Proposition 2]{k-ca}).

A proof of the following   lemma is the same as that of \cite[Lemmas 6]{k-ca-sim}. It is based on \ineq{propti}, \ineq{propt2}, the observation that
\[
|T_i(x)- \chi_i(x)| = \left|\int_{-1}^x  T_i'(u) du \right| , \quad x < x_i,
\]
\[
|T_i(x)- \chi_i(x)|  = \left|\int_x^1 T_i'(u) du\right| , \quad x > x_i,
\]
and the Dzyadyk inequality (see \eg \cite[Theorem 3, p. 262]{dz})
\be \label{dzyad}
\norm{\dn^{s+\nu} P_n^{(\nu)}}{\infty} \leq c(s,\nu)  \norm{\dn^{s} P_n}{\infty} , \quad P_n \in\Poly_n \andd s\in\R .
\ee

\begin{lemma} \label{lem5.1}
Let $1\leq i \leq n$, and let $\nu_0, \mu,\e_1, \e_2 \in\N_0$ be such that $\mu \geq c_* \max\{\nu_0,\e_1,\e_2,1\}$, where $c_*$ is some sufficiently large absolute (positive) constant. Then the polynomial
$T_i = T_{i} (n, \mu, \e_1, \e_2)$ of degree $\leq c(\mu) n$ satisfies the following inequalities for all $x\in [-1,1]$:
\[
\left| T_{i} (x) -   \chi_i(x) \right| \leq c   \psi_i(x)^\mu
\]
and
\[
\left| T_{i}^{(\nu)} (x)  \right| \leq c  |I_i|^{-\nu}   \psi_i(x)^\mu , \quad 0\leq \nu \leq \nu_0 ,
\]
where constants $c$ depend only on $\mu$.
\end{lemma}

We   note that by choosing $\e_1, \e_2$ to be $0$ or $1$ we can make polynomials $T_{i} (n, \mu, \e_1, \e_2)$ lie either above $\chi_{i-1}$ or below $\chi_i$. Indeed, recalling that $T_i(-1)=0$ and $T_i(1)=1$,
inequalities $T_{i}' (n, \mu, 1, 0)(x) (x-x_i) \geq 0$ and $T_{i}' (n, \mu, 0, 1)(x) (x_{i-1}-x) \geq 0$ immediately imply that, for all $1\leq i\leq n$,
\be \label{abovebelow}
  T_{i}  (n, \mu, 1, 0)(x) \leq \chi_i(x) \andd    T_{i}  (n, \mu, 0, 1)(x) \geq \chi_{i-1}(x)   , \quad  x\in [-1,1] .
\ee

\sect{Polynomial approximation of $w_n^{1/p}$ for $0<p<\infty$} \label{polyapprox}

\begin{theorem} \label{thm3.2}
Suppose that $w$ is a doubling weight.
For every $0<p<\infty$,  $n\in\N$ and $\nu_0\in\N$, there exists a polynomial $\Q_n \in \Poly_n$ such that, for all $x\in [-1,1]$,
\be \label{pol1}
 c w_n(x)^{1/p} \leq \Q_n (x)  \leq c w_n(x)^{1/p}
\ee
and
\be \label{pol2}
\left| \dnx^\nu   \Q_n^{(\nu)}(x) \right| \leq c   w_n(x)^{1/p}, \quad 1\leq \nu \leq \nu_0 ,
\ee
where constants $c$ depend only on $\nu_0$, $p$ and  the doubling constant of $w$.
\end{theorem}

Note that, in the case $\nu_0 =1$, \thm{thm3.2} was proved in \cite[(7.34)-(7.36)]{mt2000}.

\begin{proof}
Let $S_n(x)$ be a piecewise constant function such that
\[
S_n(x) = s_i := \sup_{u \in I_i} w_n(u)^{1/p} , \quad x\in I_i, \quad 1\leq i\leq n .
\]
Note that \ineq{ineq1.2} implies that $w_n(x)^{1/p} \leq  s_i \leq c w_n(x)^{1/p}$, for all $x\in I_i$, $1\leq i\leq n$, and so
\[
w_n(x)^{1/p} \leq S_n(x) \leq c w_n(x)^{1/p}, \quad x\in [-1,1] .
\]
We observe that
\[
S_n(x) := s_n   + \sum_{i=1}^{n-1} \left( s_i - s_{i+1} \right) \chi_i(x) , \quad x\in [-1,1],
\]
and define
\[
\Q_n(x) := s_n   + \sum_{i=1}^{n-1} \left( s_i - s_{i+1} \right) R_i(x) ,
\]
where, for each $1\leq i\leq n-1$, the polynomial $R_i$ is defined as follows
\[
R_i(x) :=
\begin{cases}
T_{i+1}(n,\mu, 0, 1) , & \mbox{\rm if} \; s_i - s_{i+1} \geq 0, \\
T_{i}(n,\mu, 1, 0) , & \mbox{\rm otherwise}, \\
\end{cases}
\]
where $\mu \in \N$ is sufficiently large (to be prescribed).
Then \ineq{abovebelow} yields
 \[
S_n(x) \leq \Q_n(x)    , \quad x\in [-1,1] ,
 \]
 which implies the left-hand   inequality in \ineq{pol1}.
Now, for each $x\in [-1,1]$, using \ineq{ineq1.3} and the fact that $\psi_{i+1} \sim \psi_i$, $1\leq i\leq n-1$,     we have
\begin{eqnarray*}
\left|\Q_n(x) -S_n(x) \right| & \leq & \sum_{i=1}^{n-1} |s_i - s_{i+1}| \cdot |R_i(x) - \chi_i(x)|
  \leq   c
\sum_{i=1}^{n-1} w_n(x_i)^{1/p}  \psi_i(x)^\mu \\
& \leq & c
w_n(x)^{1/p}  \sum_{i=1}^{n-1}  \psi_i(x)^{\mu-s/p}  \leq c w_n(x)^{1/p} ,
\end{eqnarray*}
since $\sum_{i=1}^{n-1} \psi_i(x)^{\mu-s/p} \leq c$    if $\mu-s/p \geq  2$.
Therefore,
\[
\Q_n(x) \leq S_n(x) + c w_n(x)^{1/p} \leq c w_n(x)^{1/p} ,
\]
which is the right-hand inequality  in \ineq{pol1}.

Now, recalling that $|I_i| \sim \dn(x_i)$, $1\leq i \leq n$, and using the inequality $\dnx^2 \leq 4 \dn(y) \left( |x-y|+ \dn(y)\right)$ as well as \ineq{ineq1.3} we have, for all $1\leq \nu\leq \nu_0$,
\begin{eqnarray*}
\left| \dnx^\nu   \Q_n^{(\nu)}(x) \right| &\leq& \sum_{i=1}^{n-1} \dnx^\nu  \left| s_i - s_{i+1} \right| \cdot |R_i^{(\nu)}(x)|
  \leq
c \sum_{i=1}^{n-1} \dnx^\nu  w_n(x_i)^{1/p} |I_i|^{-\nu} \psi_i(x)^\mu \\
& \leq &
c w_n(x)^{1/p} \sum_{i=1}^{n-1}   \left[ \dn(x_i) \left(|x-x_i|+\dn(x_i) \right) \right]^{\nu/2}      |I_i|^{-\nu} \psi_i(x)^{\mu-s/p} \\
& \leq &
c w_n(x)^{1/p} \sum_{i=1}^{n-1}    \psi_i(x)^{\mu-s/p-\nu/2}  \leq c w_n(x)^{1/p} ,
\end{eqnarray*}
provided $\mu-s/p-\nu/2 \geq 2$. Hence, we choose $\mu$ to be such that all conditions of \lem{lem5.1} are satisfied, and also $\mu \geq \nu_0/2 +s/p +2$.
Finally, we note that we actually constructed a polynomial $\Q_n$ of degree $\leq c(\mu)n$ that satisfies inequalities \ineq{pol1} and \ineq{pol2}.
Since $w_n(x) \sim w_m(x)$ and $\dn(x) \sim \dm(x)$ if $n\sim m$, this completes the proof for   $n\geq n_0$, for some $n_0\in\N$. For $1\leq n\leq n_0$, the statement of the theorem follows from the case $n=1$ (by setting
$\Q_1(x) :=  w_1(0)^{1/p}$, for example).
\end{proof}

 \sect{Weighted polynomial approximation: Jackson type estimates} \label{jacksonsection}

\subsection{Auxiliary results}

First, we recall the well known Whitney's theorem (see \eg \cite[Theorem 7.1, p. 195]{pp}) that states that, if  $0<p\leq \infty$, $f\in\Lp[a,b]$ and $ r\in\N$, then
\be \label{whitney}
E_{r}(f, [a,b])_p := \inf_{P_{r} \in\Poly_r} \norm{f-P_r}{\Lp[a,b]} \leq c \w_r\left( f, (b-a)/r , [a,b]\right)_p ,
\ee
where $\w_r\left( f, t, [a,b] \right)_p$ is the usual $r$th modulus of smoothness in the $\Lp$ (quasi)norm.

We also define
 the   averaged weighted modulus by
\begin{eqnarray*}
\tilde\w_\varphi^r(f, t)_{p, w} &:= & \left(\frac{1}{t}  \int_0^{t}  \int_{-1}^1       w (x) |\Delta_{h\varphi(x)}^r(f, x)|^p dx dh \right)^{1/p} \\
&=& \left(\frac{1}{t}  \int_0^{t}  \norm{ \Delta_{h\varphi}^r(f)}{p, w}^p  dh \right)^{1/p}  ,     \quad 0<p< \infty,
\end{eqnarray*}
and for convenience denote $\tilde\w_\varphi^r(f, t)_{\infty, w} := \w_\varphi^r(f, t)_{\infty, w}$.

Note that it is clear from the definition that
\[
\tilde\w_\varphi^r(f, t)_{p, w} \leq  \w_\varphi^r(f, t)_{p, w} , \quad 0<p<\infty .
\]

\begin{lemma}\label{lem4.1nnn} 
 For a doubling weight $w$,  $f\in\Lp[-1,1]$, $0<p < \infty$,   $n,r\in\N$, and any $0<\theta<1$  the following holds
 \[
 \sum_{i=1}^n w_n(x_i)\w_r(f, |J_i|, J_i)_p^p \leq c \tilde \w_\varphi^r (f, \theta/ n)_{p, w_n}^p  \leq c \w_\varphi^r (f, \theta/ n)_{p, w_n}^p     ,
 \]
 where, for every $i$, $I_i\subset J_i \subset [-1,1] $
  and $|J_i| \leq c_0 |I_i|$, and the constant $c$ depends only on $r$, $p$, $c_0$, $\theta$, and the doubling constant of $w$.
 \end{lemma}




We remark that the reason for introducing $\theta$ is that we have NOT proved the estimate
\[
\w_\varphi^r(f, \lambda/n)_{p, w_n} \leq c \w_\varphi^r(f, 1/n)_{p, w_n} , \quad p>0.
\]


\begin{proof} The proof of this lemma is rather standard and   not different from that for unweighted moduli (see \eg \cite{dly}). The main idea is the employment of the inequality (see \cite[Lemma 7.2, p. 191]{pp})
\be \label{pp-ineq}
\w_r(f, t, [a,b])_p^p \leq {c \over t} \int_0^t \int_a^b |\Delta_h^r(f, x, [a,b])|^p dx\, dh , \quad 0<p<\infty .
\ee

Note that if $J_i \supset I_i$ and $|J_i| \leq c_0 |I_i|$, then there exists $m\in\N$ depending only on $c_0$ such that $J_i$   has nonempty intersection with at most $m$ intervals $I_j$, $1\leq j\leq n$.
Since $|I_i| \sim |I_{i\pm 1}| \sim \dn(x_i)$, this implies that
 $\dnx \sim \dn(y)\sim |I_i|$ for all $x,y \in J_i$, and so $|x-y| \leq c \dnx$, for all $x,y\in J_i$.

Taking this into account and  using \ineq{pp-ineq} and \ineq{ineq1.2} we have
\begin{eqnarray*}
w_n(x_i) \w_r(f, |J_i|, J_i)_p^p &\leq& c w_n(x_i) \w_r(f, c^*|I_i| , J_i)_p^p \\
& \leq &
  c |I_i|^{-1} \int_0^{c^*|I_i| } \int_{J_i} w_n(x_i) |\Delta_h^r(f, x, J_i)|^p dx\, dh \\
&\leq&  c  \int_{J_i}  \int_0^{c^*|I_i|/\varphi(x)}  {\varphi(x) \over |I_i|}  w_n(x) |\Delta_{h\varphi(x)}^r(f, x, J_i)|^p dh \, dx ,
\end{eqnarray*}
where $0<c^* < 1$ is a constant that we will choose later.

Now,    $|I_i| \sim \dn(x) \sim \varphi(x)/n$ for $x\in J_i$, $i \in J^*$, where
\[
J^* := \left\{ 1\leq i\leq n \st  J_i \cap (I_1 \cup I_n) = \emptyset \right\} ,
\]
 and so, for $i \in J^*$, taking into account that $c^* \leq \sqrt{c^*}$ (it is a red herring for now, but  is needed because of the estimate for $i\not\in J^*$ below),  we have
 \be \label{newaux}
w_n(x_i) \w_r(f, |J_i|, J_i)_p^p \leq c n  \int_{J_i}  \int_0^{c_1 \sqrt{c^*}/n}    w_n(x) |\Delta_{h\varphi(x)}^r(f, x, J_i)|^p dh dx .
 \ee
Suppose now that $i \not\in J^*$. We recall that    $\Delta_h^r(f, x, J_i)$ is defined to be $0$ if $x\pm rh/2 \not\in J_i$ and, in particular,
$ \Delta_{h\varphi(x)}^r(f, x, J_i) = 0 $ if  $1-|x| < rh\varphi(x)/2$. Therefore, recalling that $\varphi(x)/|I_i| \leq c n \dnx/|I_i| \leq cn$, $x\in J_i$,  for each fixed $x\in J_i$, we have
\[
\int_0^{c^*|I_i|/\varphi(x)}  {\varphi(x) \over |I_i|}  w_n(x) |\Delta_{h\varphi(x)}^r(f, x, J_i)|^p dh \leq
cn \int_S w_n(x) |\Delta_{h\varphi(x)}^r(f, x, J_i)|^p dh ,
\]
where
\begin{eqnarray*}
S &:=& \left\{ h \st 0<h \leq \min\left\{ {c^*|I_i| \over  \varphi(x)},  {2(1-|x|) \over r \varphi(x)}  \right\} \right\}\\
& \subset&
\left\{ h \st 0<h \leq c_2 \min\left\{ {c^* \over  n^2  \sqrt{ 1-|x|} },  \sqrt{ 1-|x|}   \right\} \right\}
\subset
\left\{ h \st 0<h \leq c_2 \sqrt{c^*}/n \right\} .
\end{eqnarray*}
Therefore, \ineq{newaux} is valid for $i\not\in J^*$ as well (with $c_2$ instead of $c_1$). We now choose $c^*$ to be such that $\max\{c_1, c_2\} \sqrt{c^*} < \theta$.   Then
\begin{eqnarray*}
 \sum_{i=1}^n w_n(x_i)\w_r(f, |J_i|, J_i)_p^p  & \leq &
c n  \sum_{i=1}^n  \int_{J_i}  \int_0^{\theta/n}    w_n(x) |\Delta_{h\varphi(x)}^r(f, x, J_i)|^p dh dx \\
& \leq &
c n  \sum_{i=1}^n  \int_{I_i}  \int_0^{\theta/n}    w_n(x) |\Delta_{h\varphi(x)}^r(f, x)|^p dh dx \\
& \leq &
c n   \int_0^{\theta/n}  \int_{-1}^1       w_n(x) |\Delta_{h\varphi(x)}^r(f, x)|^p  dx dh \\
& \leq &
c \tilde \w_\varphi^r (f, \theta/n)_{p, w_n}^p,
\end{eqnarray*}
and the proof is complete.
\end{proof}

An analog of \lem{lem4.1nnn} in the case $p=\infty$ is the following 
result.

\begin{lemma} \label{lemmainf}
For a doubling weight $w$,  $f\in\L_\infty[-1,1]$,    $n,r\in\N$, and any $0<\theta<1$  the following holds
 \[
 \sup_{1\leq i \leq n}  w_n(x_i)\w_r(f, |J_i|, J_i)_\infty  \leq c \w_\varphi^r (f, \theta/ n)_{\infty, w_n}  ,
 \]
 where, for every $i$, $I_i\subset J_i \subset [-1,1] $
 is such that $|J_i| \leq c_0 |I_i|$, and the constant $c$ depends only on $r$,  $c_0$, $\theta$, and the doubling constant of $w$.
\end{lemma}

\begin{proof}
Let $0<c^*<1/(2r)$ be a constant that we will prescribe later, and let $1\leq i\leq n$, $x^*\in J_i$ and $h^*\in (0, c^*|I_i|]$ (note that   $h^* <1/r$)   be such that
\begin{eqnarray*}
W & := &  \sup_{1\leq j \leq n}  w_n(x_j)\w_r(f, |J_j|, J_j)_\infty =    w_n(x_i)\w_r(f, |J_i|, J_i)_\infty \\
&\leq & c w_n(x_i)\w_r(f, c^*|I_i|, J_i)_\infty \leq c w_n(x_i) |\Delta_{h^*}^r(f, x^*)|.
\end{eqnarray*}
It was shown in the proof of \lem{lem4.1nnn}, that $|x-y| \leq c\dnx$, for all $x,y\in J_i$, and so $w_n(x_i)\sim w_n(x^*)$.

Now, we set $h := h^*/\varphi(x^*)$ and consider two cases: (i) $\varphi(x^*) \geq \theta/(2n)$ and (ii) $\varphi(x^*) < \theta/(2n)$.
In the case (i), $|I_i| \sim \dn(x^*) \sim \varphi(x^*)/n$, and so $h \leq c c^*/n$ for some positive constant $c$, and we can choose $c^*$ so that $h\leq \theta/n$.
In the case (ii), since $x^* \pm rh^*/2 \in [-1,1]$, we conclude that $\varphi(x^*) \geq \sqrt{h^*/2}$, and so $h^* < \theta^2/(2n^2)$. Therefore,
$h \leq \sqrt{2 h^*}\leq \theta/n$.

Hence, for some $0<h\leq \theta/n$,
\[
W \leq c w_n(x^*) \left| \Delta_{h\varphi(x^*)}^r(f, x^*) \right| ,
\]
and so
\[
W \leq   c\sup_{0<h\leq \theta/n} \sup_{x\in [-1,1]} \left| w_n(x)  \Delta_{h\varphi(x)}^r(f, x) \right|
\leq
 c \w_\varphi^r (f, \theta/ n)_{\infty, w_n} .
\]
\end{proof}

\subsection{Jackson type estimate}

\begin{theorem} \label{jacksonthm}
Let $w$ be a doubling weight, $r\in\N$, $0<p\leq \infty$ and $f\in\Lp[-1,1]$. 
 Then, for every $n \geq r$ and $0<\ccc\leq 1$,  there exists a polynomial $P_n \in\Poly_n$ such that
\[
\norm{f-P_n}{p, w_n}   \leq c \tilde \w_\varphi^r(f, \ccc/n)_{p, w_n} \leq c \w_\varphi^r(f, \ccc/n)_{p, w_n}
\]
and
\[
\norm{ \dn^\nu P_n^{(\nu)}}{p, w_n} \leq c \tilde \w_\varphi^r(f, \ccc/n)_{p, w_n} \leq c \w_\varphi^r(f, \ccc/n)_{p, w_n} ,      \quad r\leq \nu \leq \nu_0,
\]
where constants $c$ depend only on  $r$, $\nu_0$, $p$, $\ccc$ and the doubling constant of $w$.
\end{theorem}

We remark that, in the case $p=\infty$, it is usually assumed that   $f\in\C[-1,1]$ since, otherwise,   $ \w_\varphi^r(f, 1/n)_{p, w_n} \geq c>0$, $n\in\N$, and so the assumption that $f\in\L_\infty[-1,1]$ does not make this theorem more general.

\begin{proof}
We first assume that $0<p<\infty$.
For $n\in\N$,
let $(x_i)_{i=0}^n$ be the Chebyshev partition of $[-1,1]$, and let $p_i \in\Poly_r$, $1\leq i\leq n$, be a polynomial of near best approximation of $f$ on $\I_i := I_i\cup I_{i-1}$ (with $I_0 := \emptyset$) in the $\Lp$ (quasi)norm, \ie
 \[
\norm{f-p_i}{\Lp(I_i)} \leq c E_r(f, \I_i)_p .
\]
We define $S_n$ to be a piecewise polynomial function such that $p_i  = S_n\big|_{I_i}$, $1\leq i\leq n$.

Then
\[
S_n(x) = p_n(x) + \sum_{i=1}^{n-1} \left[ p_i(x)-p_{i+1}(x) \right] \chi_i(x) .
\]
Therefore, using \ineq{ineq1.2}, \ineq{whitney} and \lem{lem4.1nnn}  we have
\begin{eqnarray*}
\norm{f-S_n}{p, w_n}^p & = & \sum_{i=1}^n \int_{I_i} w_n(x) |f(x)-S_n(x)|^p dx
  \leq
c \sum_{i=1}^n w_n(x_i) \int_{I_i}  |f(x)-p_i(x)|^p dx \\
& \leq &
c \sum_{i=1}^n w_n(x_i) \w_r(f, |I_i|, \I_i)_p^p
  \leq
c \tilde \w_\varphi^r(f, \theta/n)_{p, w_n}^p ,
\end{eqnarray*}
where $0<\theta<1$ will be chosen later.
We now define
\[
P_n(x) := p_n(x) + \sum_{i=1}^{n-1} \left[ p_i(x)-p_{i+1}(x) \right] T_{i}(x),
\]
where $T_i = T_1(n, \mu, \e_1, \e_2)$ are the polynomials from \lem{lem5.1} (note that the choice of $\e_1$ and $\e_2$ is not important; for example, we can set $\e_1=\e_2 = 0$) with a sufficiently large $\mu$ (we will prescribe it later so   that all restrictions below are satisfied).

\lem{lem5.1} now implies
\begin{eqnarray*}
\norm{  S_n -P_n }{p, w_n}^p & \leq & \int_{-1}^1 w_n(x) \left[  \sum_{i=1}^{n-1}  \left| p_i (x)-p_{i+1} (x) \right|  \cdot |\chi_i(x) - T_i(x)|   \right]^p dx \\
& \leq & c
\int_{-1}^1 w_n(x) \left[  \sum_{i=1}^{n-1}  \norm{p_i -p_{i+1}}{\infty}  \psi_i(x)^\mu   \right]^p dx \\
\end{eqnarray*}
Now, using the Lagrange interpolation formula and \cite[Theorem 4.2.7]{dl}  we have, for all $p \in \Poly_r$ and $0\leq l\leq r-1$,
\be \label{pohidna}
\norm{p^{(l)}}{\infty}  \leq c \psi_i^{-r+l+1} \norm{p^{(l)}}{\C(I_i)}    \leq c \psi_i^{-r+l+1} |I_i|^{-l-1/p} \norm{p}{\Lp(I_i)} ,
\ee
and hence
\begin{eqnarray*}
\norm{  S_n -P_n }{p, w_n}^p  & \leq &
c \int_{-1}^1 w_n(x) \left[  \sum_{i=1}^{n-1}  \norm{p_i -p_{i+1}}{\Lp(I_i)} |I_i|^{-1/p}     \psi_i(x)^{\mu-r+1}    \right]^p dx .
\end{eqnarray*}
Now, if $1\leq p < \infty$, since
$\sum_{i=1}^{n-1} \psi_i^2 \leq c$,
we have by Jensen's inequality
\[
\left( \sum_{i=1}^{n-1} |\gamma_i| \psi_i(x)^2 \right)^p \leq c \sum_{i=1}^{n-1} |\gamma_i|^p \psi_i(x)^2 \leq c \sum_{i=1}^{n-1} |\gamma_i|^p    ,
\]
and if $0<p<1$, then
\[
\left( \sum_{i=1}^{n-1} |\gamma_i| \psi_i(x)^2 \right)^p \leq   \sum_{i=1}^{n-1} |\gamma_i|^p \psi_i(x)^{2p} \leq c \sum_{i=1}^{n-1} |\gamma_i|^p  .
\]
Therefore, using \ineq{ineq1.3} we have
\begin{eqnarray*}
\norm{  S_n -P_n }{p, w_n}^p  & \leq &
c \int_{-1}^1     \sum_{i=1}^{n-1}  \norm{p_i -p_{i+1}}{\Lp(I_i)}^p |I_i|^{-1}  w_n(x) \psi_i(x)^{(\mu-r-1)p}   dx \\
& \leq &
c \int_{-1}^1     \sum_{i=1}^{n-1}  \norm{p_i -p_{i+1}}{\Lp(I_i)}^p |I_i|^{-1}  w_n(x_i) \psi_i(x)^{(\mu-r-1)p-s}   dx \\
& \leq &
c    \sum_{i=1}^{n-1}  \w_r(f, |I_i|, \I_i \cup \I_{i+1})_p^p  \  |I_i|^{-1}   w_n(x_i) \int_{-1}^1   \psi_i(x)^{(\mu-r-1)p-s}   dx.
\end{eqnarray*}
Now, if $\a\geq 2$, then $\int_{-1}^1 \psi_i(x)^\a dx \leq c |I_i|$, and so
\[
\norm{  S_n -P_n }{p, w_n}^p    \leq
c    \sum_{i=1}^{n-1} w_n(x_i) \w_r(f, |I_i|, \I_i \cup \I_{i+1})_p^p
  \leq  c   \tilde \w_\varphi^r(f, \theta/n)_{p, w_n}^p .
\]
provided $(\mu-r-1)p-s \geq 2$.

Now, note that
\[
P_n^{(\nu)}(x) = p_n^{(\nu)}(x) + \sum_{i=1}^{n-1} \sum_{l=0}^{\nu}  {\nu \choose l}     \left[ p_i^{(l)}(x)-p_{i+1}^{(l)}(x) \right] T_i^{(\nu-l)} (x) ,
\]
and so, for $r\leq \nu \leq \nu_0$ (which guarantees that $p_n^{(\nu)}\equiv 0$), we have using \lem{lem5.1} and \ineq{pohidna}
\begin{eqnarray*}
\norm{ \dn^\nu P_n^{(\nu)}}{p, w_n}^p & \leq &
\int_{-1}^1 w_n(x) \dnx^{\nu p}  \left[ \sum_{i=1}^{n-1} \sum_{l=0}^{\nu}  {\nu \choose l}     \left| p_i^{(l)}(x)-p_{i+1}^{(l)}(x) \right| \cdot \left| T_i^{(\nu-l)} (x)\right| \right]^p dx \\
& \leq &
c \int_{-1}^1 w_n(x)  \dnx^{\nu p} \left[  \sum_{i=1}^{n-1} \sum_{l=0}^{\nu}  \norm{p_i^{(l)} -p_{i+1}^{(l)}}{\infty} |I_i|^{-\nu+l}  \psi_i(x)^\mu   \right]^p dx \\
& \leq &
c \int_{-1}^1 w_n(x)  \dnx^{\nu p} \left[  \sum_{i=1}^{n-1} \sum_{l=0}^{\nu}  \norm{p_i  -p_{i+1} }{\Lp(I_i)} |I_i|^{-\nu-1/p}  \psi_i(x)^{\mu - r+l+1}   \right]^p dx \\
& \leq &
c \int_{-1}^1 w_n(x)  \dnx^{\nu p} \left[  \sum_{i=1}^{n-1}    \norm{p_i  -p_{i+1} }{\Lp(I_i)} |I_i|^{-\nu-1/p}  \psi_i(x)^{\mu - r +1}   \right]^p dx \\
& \leq &
c \int_{-1}^1 w_n(x)  \dnx^{\nu p}    \sum_{i=1}^{n-1}    \norm{p_i  -p_{i+1} }{\Lp(I_i)}^p |I_i|^{-\nu p-1}  \psi_i(x)^{(\mu - r -1)p}   dx \\
& \leq &
c \int_{-1}^1   \dnx^{\nu p}    \sum_{i=1}^{n-1}    \norm{p_i  -p_{i+1} }{\Lp(I_i)}^p |I_i|^{-\nu p-1} w_n(x_i) \psi_i(x)^{(\mu - r -1)p-s}   dx .
\end{eqnarray*}
Now, since $\dnx^2 \leq 4 \dn(x_i) \left( |x-x_i|+ \dn(x_i)\right)$ and $|I_i| \sim \dn(x_i)$,
\begin{eqnarray*}
  \norm{ \dn^\nu P_n^{(\nu)}}{p, w_n}^p
& \leq &
c \int_{-1}^1     \sum_{i=1}^{n-1}    \norm{p_i  -p_{i+1} }{\Lp(I_i)}^p   \left[ \dn(x_i) \left( |x-x_i|+ \dn(x_i)\right) \right]^{\nu p/2}  \\
&& \; \times  \;  |I_i|^{-\nu p-1} w_n(x_i) \psi_i(x)^{(\mu - r -1)p-s}   dx \\
& \leq &
c \int_{-1}^1     \sum_{i=1}^{n-1}    \norm{p_i  -p_{i+1} }{\Lp(I_i)}^p   |I_i|^{-1} w_n(x_i) \psi_i(x)^{(\mu - r -1-\nu/2)p-s}   dx ,
 \end{eqnarray*}
and exactly the same sequence of inequalities as above yields
\[
\norm{ \dn^\nu P_n^{(\nu)}}{p, w_n}^p \leq c \tilde\w_\varphi^r(f, \theta/n)_{p, w_n}^p
\]
provided $(\mu - r -1-\nu_0/2)p-s \geq 2$.
Thus, if we pick $\mu = \mu(r, \nu_0, p, s)$ so that this (the most restrictive in this proof) inequality as well as the restrictions on $\mu$ from \lem{lem5.1} are
 satisfied
then, for each $n\in\N$, we have constructed a polynomial $\tilde P_n$ of degree $< c_*  n$ with some $c_*\in\N$ depending only on $r$, $\nu_0$, $p$ and  $s$, such that
\be \label{tmp1}
\norm{ f - \tilde P_n}{p, w_n}     \leq  c  \tilde \w_\varphi^r(f, \theta/n)_{p, w_n}^p
\ee
and
\be\label{tmp2}
\norm{ \dn^\nu \tilde P_n^{(\nu)}}{p, w_n} \leq c \tilde\w_\varphi^r(f, \theta/n)_{p, w_n} , \quad r\leq \nu \leq \nu_0.
\ee

We now pick $\theta := \ccc/(2c_*)$, and conclude that this completes the proof for   $n\geq c_*$. Indeed, suppose that $n \geq c_*$. Then there exists $m\in\N$   such that
$m c_* \leq n < (m+1) c_*$. Then, for  polynomials $\tilde P_m$ of degree $<c_* m $ (which implies that $\tilde P_m \in \Poly_n$), \ineq{tmp1} and \ineq{tmp2}   hold, and
\[
\tilde \w_\varphi^r(f, \theta/m)_{p, w_m} = \tilde \w_\varphi^r(f, \ccc/(2c_* m))_{p, w_m} \leq c \tilde \w_\varphi^r(f, \ccc/n)_{p, w_m} \leq c \tilde \w_\varphi^r(f, \ccc/n)_{p, w_n},
\]
 since  $w_n(x) \sim w_m(x)$   if $n\sim m$. Also, $\dn(x) \sim \dm(x)$ if $n\sim m$, and the proof is complete for $n\geq c_*$.
 Finally, for $r\leq n\leq c_* $, the statement of the theorem follows from the case $n=r$, Whitney's inequality \ineq{whitney} and the observation that
$w_r(0)\w_r(f, 2)_{p}^p \leq c \tilde \w_\varphi^r(f, \ccc/r)_{p, w_r}^p$ (see \lem{lem4.1nnn}).
This completes the proof in the case $0<p<\infty$.

If $p=\infty$, the proof is analogous and, in fact,  simpler. For completeness, we sketch it below.

 The estimate
$\norm{f-S_n}{\infty, w_n}   \leq c \w_\varphi^r(f, \theta/n)_{\infty, w_n}$ immediately follows from \lem{lemmainf}, and for each $x\in [-1,1]$, we have
\begin{eqnarray*}
  w_n(x) \left|S_n(x) -P_n(x)\right|   & \leq &   w_n(x)   \sum_{i=1}^{n-1}  \left| p_i (x)-p_{i+1} (x) \right|  \cdot |\chi_i(x) - T_i(x)|     \\
 & \leq & c
   \sum_{i=1}^{n-1}  \norm{p_i -p_{i+1}}{\C(I_i)}  w_n(x_i)  \psi_i(x)^{\mu-r -s+1}  \\
  & \leq & c
   \sum_{i=1}^{n-1} \w_r(f, |I_i|, \I_i \cup \I_{i+1})_\infty  w_n(x_i)  \psi_i(x)^{\mu-r -s+1}  \\
  & \leq & c
\w_\varphi^r(f, \theta/n)_{\infty, w_n}   \sum_{i=1}^{n-1}      \psi_i(x)^{\mu-r -s+1}  \\
   & \leq & c
\w_\varphi^r(f, \theta/n)_{\infty, w_n}   ,
\end{eqnarray*}
provided $\mu-r -s+1\geq 2$.

Similarly, for $r\leq \nu\leq \nu_0$, as in the case $p<\infty$, we have
\begin{eqnarray*}
  \dnx^\nu w_n(x) \left| P_n^{(\nu)}(x)\right| & \leq &
 c  w_n(x) \dnx^{\nu }    \sum_{i=1}^{n-1} \sum_{l=0}^{\nu}       \left| p_i^{(l)}(x)-p_{i+1}^{(l)}(x) \right| \cdot \left| T_i^{(\nu-l)} (x)\right|   \\
& \leq &
c  w_n(x)  \dnx^{\nu  }   \sum_{i=1}^{n-1}    \norm{p_i  -p_{i+1} }{\C(I_i)} |I_i|^{-\nu}  \psi_i(x)^{\mu - r +1}    \\
& \leq &
c     \sum_{i=1}^{n-1}    \norm{p_i  -p_{i+1} }{\C(I_i)}    w_n(x_i) \psi_i(x)^{ \mu - r -s +1 - \nu/2} \\
   & \leq & c
\w_\varphi^r(f, \theta/n)_{\infty, w_n}   ,
\end{eqnarray*}
if $\mu - r -s +1 - \nu/2 \geq 2$.
This completes the proof for $n\geq c_* n$, and the rest of the proof is the same as in the case $p<\infty$ taking into account that
$w_r(0)\w_r(f, 2)_{\infty} \leq c \w_\varphi^r(f, \ccc/r)_{\infty, w_r}$.
\end{proof}

\sect{Markov-Bernstein type theorems} \label{mbsection}

\begin{lemma} \label{lem8.2}
Let $w$ be a doubling weight, $r\in\N$ and $0<p \leq \infty$. Then, for all $n\in\N$ and $P_n\in\Poly_n$,
\be\label{mbern}
\norm{  \dn^r P_n^{(r)}}{p,w_n} \leq c  \norm{  P_n}{p, w_n} .
\ee
where the constant $c$ depends only on  $r$, $p$ and the doubling constant of $w$.
\end{lemma}

\lem{lem8.2} will be used to prove an inverse theorem in the case $1\leq p\leq \infty$. However, in the case $0<p<1$, it will not be sufficient and will have to be much improved since we will need to know the dependence of the constant $c$ in \ineq{mbern} on $r$ making sure that it does not grow too fast with $r$. This will be done in Section~\ref{extra}.

\begin{proof}
First, we recall Markov-Bernstein's inequality (see \eg \cite{n}, \cite[(7.2.7)]{dt}, \cite{emn}, \cite[Lemma 4]{t}, \cite[Lemma 2.2]{djl}, \cite[(A.4.22)]{be}, \cite{op} as well \ineq{dzyad})
\be \label{markb}
\norm{  \dn^r P_n^{(r)} }{p} \leq c \norm{ P_n }{p} , \quad P_n\in\Poly_n  \andd 0<p \leq \infty ,
\ee
where $c$ depends only on $r$ and $p$.

Clearly, \ineq{mbern} is true if $r=0$.
Now,  using strong induction in $r$, we suppose   that it is true for all $0\leq l \leq r-1$.
Using \thm{thm3.2} with $\nu_0 =r+1$, for example (and noting that, in the case $p=\infty$, we take $1/p$ to be $1$ in \ineq{pol1} and \ineq{pol2}),  the Leibniz formula and \ineq{markb} we have
\begin{eqnarray*}
\norm{  \dn^r P_n^{(r)}}{p,w_n} & \leq & c \norm{  \dn^r P_n^{(r)} \Q_n}{p}
  \leq   c\norm{ \dn^r(P_n\Q_n)^{(r)} }{p} + c \sum_{l=0}^{r-1} \norm{  \dn^r P_n^{(l)} \Q_n^{(r-l)} }{p} \\
& \leq & c\norm{  \dn^r (P_n\Q_n)^{(r)} }{p} + c \sum_{l=0}^{r-1}   \norm{   \dn^l P_n^{(l)}  }{p, w_n} \\
& \leq & c  \norm{  P_n\Q_n  }{p} + c  \norm{    P_n    }{p, w_n} \leq c  \norm{   P_n   }{p, w_n} ,
\end{eqnarray*}
and so \ineq{mbern} is proved.
\end{proof}

\subsection{A refinement of \lem{lem8.2} for $0<p<1$} \label{extra}

In the proof of the inverse theorem in the case $0<p<1$, we will need to know the dependence of $c$ in \lem{lem8.2} on $r$ making sure that it does not grow too fast with $r$ (since this estimate will  be used for all $0\leq r\leq n-1$).
Hence, we need to reprove \lem{lem8.2} in the case $0<p<1$ paying particular attention to the constants in all estimates.


It is convenient to denote
\[
\delta_k(x) := \max\left\{ {\sqrt{1-x^2} \over k}, {1 \over k^2} \right\} .
\]

We start with the following lemma.

\begin{lemma} \label{oper}
 Let $0<p<1$. Then, for every $n,k\in\N$, $0\leq \mu \leq n-1$, $k\geq n/2$, and $P_n\in\Poly_n$,
\be \label{constants}
\norm{\delta_k^{\mu+1} P_n' }{p} \leq c (\mu+1) \norm{\delta_k^{\mu} P_n  }{p},
\ee
where the constant $c$ depends only on $p$ and is independent of $\mu$, $n$ and $k$.
\end{lemma}

In one form or another, \lem{oper} is   known. For example, it follows from \cite{t} and \cite{djl} (see also  \cite{op}).
However, since this result and, in particular, the exact dependence of the constant on $\mu$ is crucial in our proofs and since, as far as we know, \lem{oper}   was not  explicitly stated anywhere in the present form  we sketch its short (and rather standard) proof.

\begin{proof} It is known (see \eg \cite[(2.11)]{t} or \cite[(2.3)]{djl}) that, for any $n\in\N$, $0\leq \mu\leq n-1$  and $P_n\in\Poly_n$,
\be \label{ditzj}
\norm{\varphi^{\mu+1} P_n'}{p} \leq c_1  n (\mu+1) \norm{\varphi^{\mu} P_n}{p} , \quad c_1 = c_1(p) .
\ee
It is also well known that
\[
\norm{  P_n'}{p} \leq c_2  n^2 \norm{P_n}{p} , \quad c_2 = c_2(p) .
\]
Therefore, denoting
$\E_k := \left\{ x \st \sqrt{1-x^2} \leq 1/k \right\}$ and noting that $\delta_k(x) =1/k^2$ if $x\in \E_k$, and $\delta_k(x) =\varphi(x)/k$ if $x\in [-1,1]\setminus \E_k$, we have
\begin{eqnarray*}
2^{1-1/p} \norm{\delta_k^{\mu+1} P_n' }{p} & \leq &
   \norm{\delta_k^{\mu+1} P_n' }{\Lp(\E_k)}  + \norm{\delta_k^{\mu+1} P_n' }{\Lp([-1,1]\setminus \E_k)}    \\
& = &
  k^{-2\mu-2} \norm{ P_n' }{\Lp(\E_k)}  +  k^{-\mu-1} \norm{\varphi^{\mu+1} P_n' }{\Lp([-1,1]\setminus \E_k)}   \\
 & \leq  &
  k^{-2\mu-2} \norm{ P_n' }{p}  +  k^{-\mu-1} \norm{\varphi^{\mu+1} P_n' }{p}   \\
 & \leq  &
 c_2  k^{-2\mu-2} n^2 \norm{ P_n }{p}  + c_1  n (\mu+1)  k^{-\mu-1} \norm{\varphi^{\mu} P_n }{p}   \\
 & = &
 c_2  (n/k)^2    \norm{k^{-2\mu} P_n }{p}  + c_1  (n/k)  (\mu+1)   \norm{ \left[\varphi/k\right]^{\mu} P_n }{p}   \\
 & \leq  &
\left[  c_2  (n/k)^2 + c_1  (n/k)(\mu+1)  \right]       \norm{\delta_k^\mu  P_n }{p}     \\
& \leq  &
4(c_1+c_2) (\mu+1)  \norm{\delta_k^\mu  P_n }{p}  .
\end{eqnarray*}
\end{proof}

\begin{lemma} \label{lem5.33}
Let $w$ be a doubling weight and $0<p<1$. Then, for all
  $n,m,k\in\N$ and $\mu\in\R$ such that
  \[
  m\leq k,  \quad  n\leq k  \andd    0\leq \mu \leq n-1 ,
\]
 and $P_n \in \Poly_n$,
\[
\norm{\delta_k^{\mu+1} P_n'}{p,w_m}  \leq    (\mu+1) c_* \norm{\delta_k^\mu P_n}{p,w_m}  ,
\]
where the constant $c_*$ depends only on $p$ and the doubling constant of $w$.
\end{lemma}

\begin{proof}
First, using \thm{thm3.2} (with $\nu_0 = 1$) we let $\Q_m\in\Poly_m$ be such that
\[
c_1 w_m(x)^{1/p} \leq \Q_m(x) \leq c_2 w_m(x)^{1/p}
\]
and
\[
|\dmx \Q_m'(x)| \leq c_3 w_m(x)^{1/p},
\]
where constants $c_1$, $c_2$ and $c_3$ depend only on $p$ and the doubling constant of $w$.

Note that $P_n  \Q_m\in \Poly_{n+m-1}$ and so taking into account that $\mu \leq n-1   \leq n+m-2$ and $k \geq (n+m-1)/2$, by \lem{oper}, we have
\[
\norm{\delta_{k}^{\mu+1} (P_n\Q_m)' }{p} \leq c_4  (\mu+1) \norm{\delta_{k}^{\mu} P_n\Q_m }{p},
\]
where $c_4$ depends only on $p$.
Therefore,
\begin{eqnarray*}
\norm{\delta_k^{\mu+1} P_n'}{p,w_m} & \leq &  c_1^{-1} \norm{\delta_k^{\mu+1} P_n' \Q_m}{p}  \\
& \leq &
c_1^{-1} 2^{-1+1/p} \left( \norm{\delta_k^{\mu+1} (P_n  \Q_m)'}{p} +   \norm{\delta_k^{\mu+1} P_n \Q_m'}{p}\right) \\
& \leq &
c_1^{-1} 2^{-1+1/p} \left( c_4 (\mu+1) \norm{\delta_k^{\mu} (P_n  \Q_m) }{p} + c_3 \norm{\delta_k^{\mu+1}\dm^{-1} P_n  w_m^{1/p}}{p} \right) \\
& \leq &
c_1^{-1} 2^{-1+1/p} \left( c_2 c_4 (\mu+1) \norm{\delta_k^{\mu} P_n   }{p, w_m} + c_3 \norm{\delta_k^{\mu}P_n}{p, w_m} \right) \\
& \leq &
c_1^{-1} 2^{-1+1/p} (c_2 c_4+ c_3) (\mu+1)  \norm{\delta_k^{\mu} P_n  }{p, w_m} .
\end{eqnarray*}

\end{proof}

\begin{corollary}
Let $w$ be a doubling weight and $0<p<1$. Then, for all
  $n,m,k,r\in\N$ and $l\in\N_0$ such that $m\leq k$, $n \leq k$, $l\leq r \leq n-1$,   and $P_n \in \Poly_n$,
\[
\norm{\delta_k^{r } P_n^{(r)}}{p,w_m}  \leq     (c_*)^{r-l} {r! \over l!}   \norm{\delta_k^{l} P_n^{(l)}}{p,w_m}  ,
\]
where the constant $c_*$ depends only on $p$ and the doubling constant of $w$.
\end{corollary}

\begin{proof}
\lem{lem5.33} implies
\[
\norm{\delta_k^{j+1} P_n^{(j+1)}}{p,w_m}  = \norm{\delta_k^{j+1} \left(P_n^{(j)}\right)'}{p,w_m}   \leq (j+1) c(p) \norm{\delta_k^{j} P_n^{(j)}}{p,w_m}  , \quad \text{for all}\;\;  0\leq j \leq r-1 ,
\]
which immediately implies the statement of the corollary.
\end{proof}

Now, taking into account that   $\delta_n(x) \leq \dnx \leq 2 \delta_n(x)$, we finally get the result that we need in order to prove the inverse type theorems for $0<p<1$.

\begin{corollary} \label{maincorr}
Let $w$ be a doubling weight, $0<p<1$, $n,r\in\N$, $l\in\N_0$, $0\leq l\leq r \leq n-1$,
  and $P_n \in \Poly_n$. Then
\[
\norm{\dn^{r } P_n^{(r)}}{p,w_n}  \leq  2^l    (c_*)^{r-l} {r! \over l!}   \norm{\dn^{l} P_n^{(l)}}{p,w_n}  ,
\]
where the constant $c_*$ depends only on $p$ and the doubling constant of $w$.
\end{corollary}

 We note that exactly the same proof as above (and actually simpler since \lem{oper} is no longer needed and \ineq{ditzj} can be used)  yields the following result.

\begin{corollary} \label{maincorrphi}
Let $w$ be a doubling weight, $0<p<1$, $n,r\in\N$, $l\in\N_0$, $0\leq l\leq r \leq n-1$,
  and $P_n \in \Poly_n$. Then
\[
\norm{\varphi^{r } P_n^{(r)}}{p,w_n}  \leq  (c_*)^{r-l} {r! \over l!}  n^{r-l}  \norm{\varphi^{l} P_n^{(l)}}{p,w_n}  ,
\]
where the constant $c_*$ depends only on $p$ and the doubling constant of $w$.
\end{corollary}

\subsection{Other Markov-Bernstein type estimates in the case \mbox{$0<p<1$}}

 \begin{lemma} \label{lem5.11}
Let $0<p<1$  and $n,m,r\in\N$ be such that $m\leq n$, and suppose that
  $w$ is a doubling weight from the class $\Wdg_\Lambda$ with $\gamma \leq rp$.

 Then, for any $\kappa >0$,  there exists a positive constant $\ccc$ depending only on $\kappa$, $r$,  $p$, $\Lambda$, and the doubling constant of $w$,
 such that, for any $P_m \in \Poly_m$ and $0<t \leq 1/m$,
 \be \label{i7.8}
 \w_\varphi^r (P_m, \ccc t)_{p, w_n} \leq \kappa  \left({n \over m}\right)^{\delta/p} (t m)^r \norm{\dm^r P_m^{(r)}}{p, w_m} .
 \ee
 \end{lemma}

The following corollary is an immediate consequence of  \lem{lem5.11} and \cor{maincorr} with $l=0$ (by setting $\kappa := \left[   (c_*)^{r}r!\right]^{-1}$, where $c_*$ is the constant from \cor{maincorr}).

 \begin{corollary} \label{lem5.11cor}
Let $0<p<1$  and $n,m,r\in\N$ be such that $m\leq n$, and suppose that
  $w$ is a doubling weight from the class $\Wdg_\Lambda$ with $\gamma \leq rp$.

 Then,   there exists a positive constant $\ccc$ depending only on $r$,  $p$, $\Lambda$, and the doubling constant of $w$,
 such that, for any $P_m \in \Poly_m$ and $0<t \leq 1/m$,
 \be \label{i7.8cor}
 \w_\varphi^r (P_m, \ccc t)_{p, w_n} \leq   \left({n \over m}\right)^{\delta/p} (t m)^r \norm{  P_m }{p, w_m} .
 \ee
 \end{corollary}

 \begin{proof}[Proof of \lem{lem5.11}] The method of the proof is rather standard (see \eg \cites{t, dly, dhi}).
 Suppose that    $h \leq  \ccc t \leq \ccc/m$, where $\ccc$ is a positive constant that we will choose later.
Using Taylor's expansion of $P_m$ we have
\begin{eqnarray} \label{taylor}
\Delta_{h\varphi(x)}^r (P_m, x) &=& \sum_{i=0}^r {r \choose i} (-1)^{r-i} P_m(x + (i-r/2)h \varphi(x))   \\ \nonumber
&=& \sum_{i=0}^r {r \choose i} (-1)^{r-i} \sum_{j=0 }^{m-1} {(i-r/2)^j h^j \over j!} \ \varphi(x)^j P_m^{(j)}(x) \\ \nonumber
&=&
\sum_{j=0 }^{m-1}  \varphi(x)^j P_m^{(j)}(x) {h^j \over j!}   \sum_{i=0}^r {r \choose i} (-1)^{r-i}  (i-r/2)^j      \\ \nonumber
&=&
\sum_{j=0 }^{m-1}  \varphi(x)^j P_m^{(j)}(x) {h^j \over j!}  \  \Delta_{1}^r \left((\cdot)^j, 0\right) .
\end{eqnarray}
Recall now that, if $g^{(r)}$ is continuous on $[x-r \mu/2, x+ r\mu/2]$ then, for some $\xi \in (x-r \mu/2, x+ r\mu/2)$,
\[
\Delta_\mu^r (g,x) = \mu^r g^{(r)}(\xi) .
\]
This implies
\be \label{estimated}
|\Delta_{1}^r \left((\cdot)^j, 0\right)| \leq
\begin{cases}
0, & \mbox{\rm if } 0\leq j \leq r-1 ,\\
\ds {j! \over (j-r)! } (r/2)^{j-r} , & \mbox{\rm if } j \geq r.
\end{cases}
\ee
Also, since $w\in\Wdg_\Lambda$,
\be \label{notneed}
w_n(x) \varphi(x)^{\gamma} \leq \Lambda n^\delta m^{\gamma-\delta} \dm(x)^{\gamma}   w_m(x)  ,
\ee
and    taking into account that $jp \geq rp \geq \gamma$, $r\leq j\leq m-1$, and $\varphi(x) \leq m \dm(x)$, we have
\begin{eqnarray*}
\lefteqn{ \norm{\Delta_{h\varphi}^r (P_m)}{p, w_n}^p}\\
  & = &
\int_{-1}^1 w_n(x) |\Delta_{h\varphi(x)}^r (P_m, x)|^p dx \\
& \leq &
\sum_{j=r }^{m-1}  \left( {h^j \over (j-r)!}(r/2)^{j-r} \right)^p  \int_{-1}^1 w_n(x) \left|\varphi(x)^j P_m^{(j)}(x)\right|^p dx \\
& \leq &
\Lambda \sum_{j=r }^{m-1}  \left( {h^j \over (j-r)!}(r/2)^{j-r} \right)^p  \int_{-1}^1
  n^\delta m^{\gamma-\delta} \dm(x)^{\gamma} \varphi(x)^{jp-\gamma}   w_m(x)
\left| P_m^{(j)}(x)\right|^p dx \\
& \leq &
\Lambda \left({n \over m}\right)^\delta  \sum_{j=r }^{m-1}  \left( {(hm)^j \over (j-r)!}(r/2)^{j-r} \right)^p  \int_{-1}^1 w_m(x) \left|\dm(x)^j P_m^{(j)}(x)\right|^p dx .
\end{eqnarray*}

It follows from \cor{maincorr} that, for some constant $c_1$ that depends only on $p$ and   the doubling constant of $w$,
\begin{eqnarray*}
\lefteqn{ \int_{-1}^1 w_m(x) \left|\dm(x)^j P_m^{(j)}(x)\right|^p dx   =   \norm{\dm^j P_m^{(j)}}{p, w_m}^p} \\
&\leq&  \left( 2^r (c_1)^{j-r} {j! \over r!}\right)^p  \norm{\dm^r P_m^{(r)}}{p, w_m}^p  , \quad   r\leq j\leq m-1 .
\end{eqnarray*}

Therefore, recalling that $h   \leq \ccc/m$, we have
\begin{eqnarray*}
 \norm{\Delta_{h\varphi}^r (P_m)}{p, w_n}^p
& \leq &
\Lambda \left({n \over m}\right)^\delta \sum_{j=r }^{m-1}  \left( {(hm)^j j! \over (j-r)! r!}(r/2)^{j-r}  2^r (c_1)^{j-r}      \right)^p  \norm{\dm^r P_m^{(r)}}{p, w_m}^p \\
& \leq &
\Lambda 2^{rp}  \left({n \over m}\right)^\delta   \norm{\dm^r P_m^{(r)}}{p, w_m}^p (hm)^{rp}  \sum_{j=r }^{m-1}  \left(    {j \choose r}   (\ccc r c_1 /2)^{j-r} \right)^p .
\end{eqnarray*}
Now, if $\ccc \leq  1/(rc_1)$, noting that  $\sum_{j=r }^{\infty}  \left(1/2\right)^{(j-r)p} \left[{j \choose r}\right]^{p} = c_2$, where $c_2$ depends only on $r$ and $p$, we conclude that
\[
\w_\varphi^r (P_m, \ccc t)_{p, w_n} \leq 2^{r} \left(\Lambda c_2 \right)^{1/p}  \ccc^r   \left({n \over m}\right)^{\delta/p} (t m)^r \norm{\dm^r P_m^{(r)}}{p, w_m}.
\]
Hence, if we guarantee that $\ccc$ is such that $2^{r} \left(\Lambda c_2 \right)^{1/p}  \ccc^r \leq \kappa$, then
\[
\w_\varphi^r (P_m, \ccc t)_{p, w_n} \leq \kappa  \left({n \over m}\right)^{\delta/p} (t m)^r \norm{\dm^r P_m^{(r)}}{p, w_m},
\]
and the proof is complete if we pick $\ccc := \min\left\{1/(rc_1), \kappa^{1/r} \left(\Lambda c_2 \right)^{-1/(rp)}/2 \right\}$.
 \end{proof}

Note now that if the same weight $w_n$ is used on both sides of \ineq{i7.8} (\ie $m=n$), then there is no need to use \ineq{notneed} in the proof of \lem{lem5.11}.
 Also, one can keep using $\varphi^{jp}$ and not replace it by $(n\dn)^{jp}$, and use \cor{maincorrphi} instead of \cor{maincorr} in order to estimate
 $\norm{\varphi^j P_n^{(j)}}{p, w_n}$.
The following result  that is proved   using an idea from  \cite{dhi}  will be used in the last section to show the equivalence of the moduli and  certain  realization functionals. Even though the proof is very similar to that of \lem{lem5.11} we sketch it below for completeness.

\begin{lemma}  \label{lem5.11new}
Let $w$ be a doubling weight, $0<p<1$  and $n,r\in\N$.
Then,   there exists a positive constant $\ccc$ depending only on  $r$,  $p$ and  the doubling constant of $w$,
 such that, for any $P_n \in \Poly_n$ and $0< h \leq t \leq \ccc/n$,
 \[  
\left( 1/2 \right)^{1/p}  h^{r} \norm{\varphi^r P_n^{(r)}}{p, w_n}    \leq   \norm{\Delta_{h\varphi}^r (P_n)}{p, w_n}\leq (3/2)^{1/p} h^{r} \norm{\varphi^r P_n^{(r)}}{p, w_n} ,
\]
and so
\[
\left( 1/2 \right)^{1/p}  t^r \norm{\varphi^r P_n^{(r)}}{p, w_n} \leq \w_\varphi^r (P_n,   t)_{p, w_n}  \leq (3/2)^{1/p} t^r \norm{\varphi^r P_n^{(r)}}{p, w_n} .
\]
 \end{lemma}

\begin{proof} The beginning of the proof is similar to that of \lem{lem5.11}.
We suppose that    $h \leq   t \leq \ccc/n$, where $\ccc$ is a positive constant that we will choose later.
Then using \ineq{taylor} and \ineq{estimated}, and taking into account that
$\Delta_{1}^r \left((\cdot)^r, 0\right) = r!$, we have
\begin{eqnarray*}
\norm{\Delta_{h\varphi}^r (P_n) - h^r\varphi^r P_n^{(r)}}{p, w_n}^p  \leq
 \sum_{j=r+1}^{n-1}   \left( {h^j \over (j-r)!} (r/2)^{j-r} \right)^p
 \int_{-1}^1 w_n(x) \left|\varphi(x)^j P_n^{(j)}(x)\right|^p dx .
\end{eqnarray*}
Using \cor{maincorrphi} we conclude that, for some constant $c_1$ that depends only on $p$ and   the doubling constant of $w$,
\begin{eqnarray*}
\lefteqn{ \int_{-1}^1 w_n(x) \left|\varphi(x)^j P_n^{(j)}(x)\right|^p dx = \norm{\varphi^j P_n^{(j)}}{p, w_n}^p}\\
 &\leq&  \left(  (c_1)^{j-r} {j! \over r!}\right)^p n^{(j-r)p}  \norm{\varphi^r P_n^{(r)}}{p, w_n}^p  , \quad   r+1\leq j\leq n-1 .
\end{eqnarray*}
 Therefore, recalling that $h   \leq \ccc/n$, we have
\begin{eqnarray*}
\norm{\Delta_{h\varphi}^r (P_n) - h^r\varphi^r P_n^{(r)}}{p, w_n}^p  &\leq&
 \sum_{j=r+1 }^{n-1}  \left( {h^j \over (j-r)!} (r/2)^{j-r} (c_1)^{j-r} {j! \over r!} \right)^p  n^{(j-r)p} \norm{\varphi^r P_n^{(r)}}{p, w_n}^p \\
 & \leq &
 h^{rp} \norm{\varphi^r P_n^{(r)}}{p, w_n}^p   \sum_{j=r+1 }^{n-1}  \left(     (\ccc rc_1/2)^{j-r}   {j \choose r} \right)^p \\
  & \leq &
 h^{rp} \norm{\varphi^r P_n^{(r)}}{p, w_n}^p   (\ccc rc_1/2)^p     \sum_{j=r+1 }^{n-1}  \left(     (\ccc rc_1/2)^{j-r-1}   {j \choose r} \right)^p .
\end{eqnarray*}
Now, if $\ccc \leq  1/(rc_1)$, then $\sum_{j=r+1}^{\infty}  \left(1/2\right)^{(j-r-1)p} \left[{j \choose r}\right]^{p} = c_2$, where $c_2$ depends only on $r$ and $p$, and
if $\ccc \leq 2(2c_2)^{-1/p}(rc_1)^{-1}$, then we get
\[
\norm{\Delta_{h\varphi}^r (P_n) - h^r\varphi^r P_n^{(r)}}{p, w_n}^p \leq \frac 12  h^{rp} \norm{\varphi^r P_n^{(r)}}{p, w_n}^p .
\]
Therefore, if we set $\ccc := \min\{1/(rc_1), 2(2c_2)^{-1/p}(rc_1)^{-1}\}$, then we get
\[
\frac 12 h^{rp} \norm{\varphi^r P_n^{(r)}}{p, w_n}^p    \leq   \norm{\Delta_{h\varphi}^r (P_n)}{p, w_n}^p \leq \frac 32 h^{rp} \norm{\varphi^r P_n^{(r)}}{p, w_n}^p .
\]
\end{proof}

 \sect{Weighted polynomial approximation: inverse theorems} \label{inversesection}

\subsection{Auxiliary results}

\begin{lemma} \label{lem5.2}
If $w$ is a doubling weight, $ 0< p\leq \infty$, $f\in\Lp[-1,1]$,
  $n,r\in\N$, $c_* >0$, and $t \leq c_*/n$,  then
\[
\w_\varphi^r(f, t)_{p, w_n}    \leq     c \norm{f}{p, w_n}  ,
\]
where $c$ depends only on  $r$, $c_*$, $p$, and the doubling constant of $w$.
\end{lemma}

\begin{proof}
First, we recall that  $\Delta_{h\varphi(x)}^r(f,x) = 0$ if $x\not\in \Dom_{rh/2}$, where
\[
\Dom_\lambda := \left\{ x \st x\neq \pm 1 \andd x \pm \lambda \varphi(x) \in [-1,1] \right\}  = \left\{ x \st |x| \leq {1-\lambda^2 \over 1+ \lambda^2} \right\},
\]
and hence, for $0<p<\infty$,  
\begin{eqnarray*}
\w_\varphi^r(f, t)_{p, w_n}^p & \leq & c \sup_{0<h\leq t}  \int_{\Dom_{rh/2}} w_n (x)
\left( \sum_{i=0}^r {r \choose i} |f(x+(i-r/2)h \varphi(x)) | \right)^p \, dx\\
& \leq &
c \sup_{0<h\leq t}  \sum_{i=0}^r \int_{\Dom_{rh/2}} w_n (x)
\left|f(x+(i-r/2)h \varphi(x))  \right|^p \, dx .
\end{eqnarray*}
It is clear that, if $h\leq t \leq c/n$, then
$h \varphi(x) \leq c \dnx$.
Therefore,
if $y_i(x) := x+(i-r/2)h \varphi(x)$, $0\leq i\leq r$,  then $|x-y_i(x)| \leq  rh\varphi(x)/2  \leq c \dnx$ and \ineq{ineq1.2} implies that
$w_n(x)\sim w_n(y_i(x))$.
Hence,
\begin{eqnarray*}
\w_\varphi^r(f, t)_{p, w_n}^p  & \leq &
c \sup_{0<h\leq t}  \sum_{i=0}^r \int_{\Dom_{rh/2}} w_n (y_i(x))
\left|f(y_i(x))  \right|^p \, dx \\
& \leq & c    \int_{-1}^1 w_n (y) \left|f(y)  \right|^p \, dy  \leq c \norm{f}{p, w_n}^p  .
\end{eqnarray*}
In the case $p=\infty$, the needed modifications in the proof are obvious.
\end{proof}

\begin{lemma} \label{lem8.4}
Let $w$ be a doubling weight,
  $n,r\in\N$, $c_*>0$, $t \leq c_*/n$, $1\leq p\leq \infty$. If $f$ has the $(r-1)$st locally absolutely continuous derivative on $(-1,1)$ and $\norm{\varphi^r f^{(r)}}{p} < \infty$, then
\[
\w_\varphi^r(f, t)_{p, w_n} \leq c t^r \norm{  \varphi^r f^{(r)}}{p, w_n} ,
\]
where $c$ depends only on  $r$, $c_*$, $p$, and the doubling constant of $w$.
\end{lemma}

We remark that it is well known that, in general, \lem{lem8.4} is not true for $0<p<1$ and, in fact, one can show that, for every $M\in\R$ and $n\in\N$, there exists an absolutely continuous function $f$ on $[-1,1]$ such that
$E_n(f, [-1,1])_p  > M\norm{f'}{p}$.

\begin{proof}
If $f$ has the $(r-1)$st absolutely continuous derivative, then
\[
  \Delta_h^r(f, x) = \int_{-h/2}^{h/2}  \dots \int_{-h/2}^{h/2} f^{(r)}(x + t_1 + \dots + t_r) dt_r \dots dt_1 .
\]
In the case $1\leq p <\infty$, if $h\leq c/n$, we have
\begin{eqnarray*}
\lefteqn{ \left( \int_{-1}^1  w_n(x) |  \Delta_{h\varphi(x)}^r(f, x)|^p dx\right)^{1/p} }\\
& \leq &
\left(  \int_{\Dom_{rh/2}}  \left[ \int_{-h\varphi(x)/2}^{h\varphi(x)/2} \dots \int_{-h\varphi(x)/2}^{h\varphi(x)/2} w_n^{1/p}(x)   |f^{(r)}(x + t_1 + \dots + t_r   )| dt_r \dots dt_1\right]^p dx  \right)^{1/p} \\
& \leq &
c\left(  \int_{\Dom_{rh/2}}  \left[ \int_{-h\varphi(x)/2}^{h\varphi(x)/2} \dots \int_{-h\varphi(x)/2}^{h\varphi(x)/2} w_n^{1/p}(x + t_1 + \dots + t_r)  \right. \right.\\
 &&  \left. \left. \; \times\; |f^{(r)}(x + t_1 + \dots + t_r   )| dt_r \dots dt_1\right]^p dx  \right)^{1/p} .
\end{eqnarray*}
By H\"{o}lder's inequality, for each $u$ satisfying $-1 < x+u - h\varphi(x)/2 < x+u + h\varphi(x)/2 <1$, we have
\begin{eqnarray*}
  \int_{-h\varphi(x)/2}^{h\varphi(x)/2} w_n^{1/p}(x + u + t_r)   |f^{(r)}(x + u  + t_r   )| dt_r   & = &
\int_{x+u-h\varphi(x)/2}^{x+u+h\varphi(x)/2} w_n^{1/p}(v)   |f^{(r)}(v)| dv  \\
& \leq &
\norm{ w_n^{1/p} \varphi^r f^{(r)}}{\Lp(\A(x,u))}  \norm{\varphi^{-r}}{\L_{p'} (\A(x,u))} ,
\end{eqnarray*}
where $1/p+1/p'=1$ and
\[
\A(x,u) := \left[ x+u-h\varphi(x)/2, x+u+h\varphi(x)/2 \right] .
\]
The needed estimate now follows from
\begin{eqnarray} \label{mess1}
\lefteqn{ \int_{\Dom_{rh/2}}  \left[ \int_{-h\varphi(x)/2}^{h\varphi(x)/2} \dots \int_{-h\varphi(x)/2}^{h\varphi(x)/2}  \norm{\varphi^{-r}}{\L_{p'} (\A(x,t_1+\dots+t_{r-1}))} \right. } \\ \nonumber
&& \left. \times
\norm{ w_n^{1/p} \varphi^r f^{(r)}}{\Lp(\A(x,t_1+\dots+t_{r-1}))}     dt_{r-1}  \dots dt_1\right]^p dx  \leq c h^{rp}\norm{ w_n^{1/p} \varphi^r f^{(r)}}{p}^p , 
\end{eqnarray}
where $1 \leq p < \infty$.
In the case $p=\infty$, an analogous sequence of estimates yields
\begin{eqnarray} \label{mess2}
\lefteqn{ \sup_{x\in \Dom_{rh/2}}   \int_{-h\varphi(x)/2}^{h\varphi(x)/2} \dots \int_{-h\varphi(x)/2}^{h\varphi(x)/2}  \norm{\varphi^{-r}}{\L_{1} (\A(x,t_1+\dots+t_{r-1}))}   } \\ \nonumber
&&   \times
\norm{ w_n  \varphi^r f^{(r)}}{\L_\infty(\A(x,t_1+\dots+t_{r-1}))}     dt_{r-1}  \dots dt_1    \leq c h^{r}\norm{ w_n \varphi^r f^{(r)}}{\infty}    .
\end{eqnarray}
Note that, in the case $r=1$, estimates \ineq{mess1} and \ineq{mess2} are understood, respectively, as
\begin{eqnarray} \label{mess3}
 \int_{\Dom_{h/2}}      \norm{\varphi^{-1}}{\L_{p'} (\A(x,0))}^p
\norm{ w_n^{1/p} \varphi  f'}{\Lp(\A(x,0))}^p   dx  \leq c h^{p}\norm{ w_n^{1/p} \varphi f'}{p}^p , \quad 1\leq p < \infty ,
\end{eqnarray}
and
\begin{eqnarray} \label{mess4}
  \sup_{x\in \Dom_{h/2}}    \norm{\varphi^{-1}}{\L_{1} (\A(x,0))}
\norm{ w_n  \varphi  f' }{\L_\infty(\A(x,0))}      \leq c h \norm{ w_n \varphi  f' }{\infty} , \quad p=\infty .
\end{eqnarray}
Estimates \ineq{mess1}-\ineq{mess4}  were proved in \cite{kls-ca} (see (4.2)-(4.4) there with $r=0$, variable ``$k$'' replaced by ``$r$'', $g^{(r)}$ replaced by $w_n^{1/p} f^{(r)}$ with $1/\infty := 1$, and noting that
$\Dom_\lambda$ in \cite{kls-ca} is actually $\Dom_{\lambda/2}$ in the current paper).
\end{proof}

\subsection{Inverse  theorem: the case $1\leq p\leq \infty$}

Recall the following notation that was used in the introduction
\[ 
\gp :=
\begin{cases}
p , & \mbox{\rm if }   p < \infty ,\\
1,  & \mbox{\rm if }  p = \infty .
\end{cases}
\]

\begin{theorem} \label{conversethm}
Let $r\in\N$, $1\leq p \leq \infty$, and $f\in\Lp[-1,1]$.
Suppose that $w$ is  a doubling weight from the class $\Wdg_\Lambda$ with $\gamma\leq r \gp$.
Then
\[ 
\w_\varphi^r (f, n^{-1})_{p, w_n}  \leq  {c \over n^{r-\delta/\gp}}   \sum_{k=1}^{n}    k^{r-1 -\delta/\gp }     E_{k}(f)_{p,w_{k}} ,
\]
where
the constant $c$ depends only on $r$, $p$, $\delta$, $\gamma$, $\Lambda$, and the doubling constant of the weight $w$.
\end{theorem}

Taking into account that any doubling weight belongs to the class $\W^{1,1}_1$   (see Remark~\ref{remcr}) and that $\gamma = 1 \leq r \gp$, for all $r\in\N$ and $1\leq p\leq \infty$, we immediately get the following corollary of \thm{conversethm}.

\begin{corollary} \label{conversethmcor}
Let $w$ be a doubling weight, $r\in\N$, $1\leq p \leq \infty$, and $f\in\Lp[-1,1]$. Then
\[ 
\w_\varphi^r (f, n^{-1})_{p, w_n}  \leq  {c \over n^{r-1/\gp}}   \sum_{k=1}^{n}    k^{r-1 -1/\gp }     E_{k}(f)_{p,w_{k}} ,
\]
 the constant $c$ depends only on $r$, $p$ and the doubling constant of $w$.
\end{corollary}

\begin{remark}
Since any weight  that satisfies the $A^*$ property is in $\W^{0,0}$ (see Remark~\ref{remast}), it immediately follows from \thm{conversethm} that, for $A^*$ weights $w$, we have
 \[
\w_\varphi^r (f, n^{-1})_{p, w_n}  \leq  {c \over n^{r}}   \sum_{k=1}^{n}    k^{r-1 }     E_{k}(f)_{p,w_{k}} , \quad 1\leq p\leq \infty .
\]
In the case $p=\infty$, this is the ``inverse'' part of \cite[Theorem 1.3]{mt2001}.
\end{remark}

\begin{proof}[Proof of \thm{conversethm}]
Let $P_n^* \in \Poly_n$ denote a polynomial of (near) best approximation to $f$ with weight $w_n$, \ie
\[
c \norm{f-P_n^*}{p,w_n} \leq   \inf_{P_n\in\Poly_n} \norm{f-P_n}{p,w_n} =  E_{n}(f)_{p,w_n} .
\]
We let $N\in\N$ be such that $2^N \leq n < 2^{N+1}$ and denote $m_j := 2^j$.
Then recalling that $w_n(x)\sim w_m(x)$ if $n \sim m$, and using \lem{lem5.2} we have
\begin{eqnarray*}
\w_\varphi^r (f, n^{-1})_{p, w_n} & \leq & \w_\varphi^r (f, 2^{-N})_{p, w_n} \\
& \leq & \w_\varphi^r (f - P_{m_N}^*, 2^{-N})_{p, w_n} + \w_\varphi^r (P_{m_N}^*, 2^{-N})_{p, w_n} \\
& \leq & c \norm{f - P_{m_N}^*}{p, w_{m_N}} + \w_\varphi^r (P_{m_N}^*, 2^{-N})_{p, w_{m_N}} \\
& \leq &
 c E_{m_N}(f)_{p,w_{m_N}} + \w_\varphi^r (P_{m_N}^*, 2^{-N})_{p, w_{m_N}}.
\end{eqnarray*}

Now, the fact that $w\in \Wdg_\Lambda$ implies (see \ineq{touse})
\[
w_{m_N}(x) \varphi(x)^{\gamma} \leq  \Lambda  m_N^\delta    m_j^{\gamma-\delta} \dlt_{m_j}(x)^{\gamma}  w_{m_j}(x)  , \quad  0\leq j \leq N .
\]

Hence, using
\[
P_{m_N}^* =  P_1^* + \sum_{j=0}^{N-1} (P_{m_{j+1}}^* - P_{m_{j}}^*)
\]
and Lemmas~\ref{lem8.4} and \ref{lem8.2}  we have
\begin{eqnarray*}
\w_\varphi^r (P_{m_N}^*, 2^{-N})_{p, w_{m_N}}  &\leq &
 \sum_{j=0}^{N-1} \w_\varphi^r  \left( P_{m_{j+1}}^* - P_{m_{j}}^*, 2^{-N}\right)_{p, w_{m_N}}  \\
 &\leq &
 c \sum_{j=0}^{N-1}   2^{-Nr } \norm{w_{m_N}^{1/\gp} \varphi^r \left(P_{m_{j+1}}^* - P_{m_{j}}^*\right)^{(r)}}{p}   \\
  &\leq &
  c \sum_{j=0}^{N-1}   2^{-Nr }  \norm{  m_N^{\delta/\gp}    m_j^{(\gamma-\delta)/\gp}  \dlt_{m_j}^{\gamma/\gp} \varphi^{r-\gamma/\gp}    w_{m_j}^{1/\gp}      \left(P_{m_{j+1}}^* - P_{m_{j}}^*\right)^{(r)}}{p}.
\end{eqnarray*}
Since $r-\gamma/\gp \geq 0$ and $\varphi \leq m_j \dlt_{m_j}$, this yields

\begin{eqnarray*}
\w_\varphi^r (P_{m_N}^*, 2^{-N})_{p, w_{m_N}}
 &\leq &
  c \sum_{j=0}^{N-1}   2^{-Nr }  \norm{  m_N^{\delta/\gp}  m_j^{r-\delta/\gp}     \dlt_{m_j}^{r}     w_{m_j}^{1/\gp}      \left(P_{m_{j+1}}^* - P_{m_{j}}^*\right)^{(r)}}{p}   \\
  &\leq &
  c \sum_{j=0}^{N-1}   2^{-(N-j)(r -\delta/\gp) }  \norm{  \dlt_{m_j}^{r}   w_{m_j}^{1/\gp}      \left(P_{m_{j+1}}^* - P_{m_{j}}^*\right)^{(r)}}{p}   \\
  &\leq &
  c \sum_{j=0}^{N-1}   2^{-(N-j)(r -\delta/\gp) }  \norm{    w_{m_j}^{1/\gp}     \left(P_{m_{j+1}}^* - P_{m_{j}}^*\right) }{p}   \\
  &\leq &
  c \sum_{j=0}^{N-1}   2^{-(N-j)(r -\delta/\gp) }    E_{m_j}(f)_{p,w_{m_j}} .
\end{eqnarray*}
Therefore,
\[
\w_\varphi^r (f, n^{-1})_{p, w_n} \leq c \sum_{j=0}^{N}   2^{-(N-j)(r -\delta/\gp) }    E_{m_j}(f)_{p,w_{m_j}} ,
 \]
and so
\begin{eqnarray*}
\w_\varphi^r (f, n^{-1})_{p, w_n} &\leq&  {c \over n^{r-\delta/\gp}}  \sum_{j=0}^{N}   2^{j(r -\delta/\gp) }    E_{m_j}(f)_{p,w_{m_j}} \\
&\leq &
{c \over n^{r-\delta/\gp}} \left(E_{1}(f)_{p,w_{1}} +  \sum_{j=1}^{N}  \sum_{k=m_{j-1}+1}^{m_{j}} k^{r-1 -\delta/\gp }     E_{k}(f)_{p,w_{k}} \right) \\
& \leq &
{c \over n^{r-\delta/\gp}}   \sum_{k=1}^{n}    k^{r-1 -\delta/\gp}     E_{k}(f)_{p,w_{k}} .
\end{eqnarray*}
\end{proof}

We have the following immediate corollaries of Theorems~\ref{jacksonthm} and \ref{conversethm}. 

\begin{corollary}
Let $r\in\N$,   $1\leq p \leq \infty$ and $f\in\Lp[-1,1]$. Suppose that
 $w$ is a doubling weight from the class $\Wdg$ with $\gamma \leq r\gp$.    Then, for $0<\a< r-\delta/\gp$, we have
\[
E_n(f, [-1,1])_{p, w_n} = O(n^{-\a}) \iff  \w_\varphi^r(f, n^{-1})_{p, w_n} = O(n^{-\a}) .
\]
\end{corollary}

Again, taking into account that any doubling weight belongs to the class $\W^{1,1}_1$   and that $1 \leq r \gp$, for all $r\in\N$ and $1\leq p\leq \infty$, we   get the following corollaries
(or one can obtain them  as a consequence of  \thm{jacksonthm} and \cor{conversethmcor}).

\begin{corollary}[$1 < p < \infty$ and all doubling weights]
Let $w$ be a doubling weight, $r\in\N$,   $1 < p< \infty$ and $f\in\Lp[-1,1]$.   Then, for $0<\a< r-1/p$, we have
\[
E_n(f, [-1,1])_{p, w_n} = O(n^{-\a}) \iff  \w_\varphi^r(f, n^{-1})_{p, w_n} = O(n^{-\a}) .
\]
\end{corollary}

Clearly, this corollary is also valid for $p=1$ and $p=\infty$. However, since $r - 1/\gp = r-1$ in both of these cases it seems more natural to state them in the following form   replacing $r-1$ with $r$.

\begin{corollary}[$p=1$ or $ p = \infty$, and   all doubling weights] \label{cor5.6}
Let $w$ be a doubling weight, $r\in\N$,   $p=1$ or $p=\infty$, and $f\in\Lp[-1,1]$.   Then, for $0<\a< r $, we have
\[
E_n(f, [-1,1])_{p, w_n} = O(n^{-\a}) \iff  \w_\varphi^{r+1}(f, n^{-1})_{p, w_n} = O(n^{-\a}) .
\]
\end{corollary}

In the case $p=\infty$, \cor{cor5.6} was proved in \cite{mt2001} (with $\w_\varphi^{r+2}$ instead of $\w_\varphi^{r+1}$). Also,
it was shown in \cite[p. 183]{mt2001} that, in the case $r=1$ and $p=\infty$, \cor{cor5.6} is no longer true if $\w_\varphi^{r+1}$ is replaced by $\w_\varphi^{r}$.

\subsection{Inverse  theorem: the case $0< p <1$}

 \begin{theorem} \label{conversethmplessone}
Let $0<p<1$,  $f\in\Lp[-1,1]$ and let $r\in\N$, and suppose that $w$ is a doubling weight from the class $\Wdg_\Lambda$ with $\gamma \leq rp$.
 Then
\[ 
\w_\varphi^r (f, \ccc n^{-1})_{p, w_n}  \leq  {c \over n^{r-\delta/p}}   \left( \sum_{k=1}^{n}    k^{rp-\delta-1}     E_{k}(f)_{p,w_{k}}^p \right)^{1/p},
\]
where $\ccc$ is the constant from \cor{lem5.11cor}, and   the constant  $c$   depends  only on $r$, $p$ and the doubling constant of $w$.
\end{theorem}

We now recall that any doubling weight $w$ belongs to the class $\Wdg$ with $(\delta, \gamma) \in\Upsilon$. In particular, $w$ belongs to the class $\W^{\delta_0, \gamma_0}$ with $\gamma_0 := \min\{rp, 1\}$ and $\delta_0 := 2-\gamma_0$. Hence, we get a corollary of \thm{conversethmplessone} for  all $r\in\N$, $0<p<1$, and  doubling weights $w$ with $\delta_0  = 2-\min\{rp, 1\}$. However, in the case $rp \leq 1$ this corollary is useless since
the resulting inequality
\[
\w_\varphi^r (f, \ccc n^{-1})_{p, w_n}  \leq   c  n^{2(1/p-r)}   \left( \sum_{k=1}^{n}    k^{2rp-3}     E_{k}(f)_{p,w_{k}}^p \right)^{1/p}
\]
 simply means that $\w_\varphi^r (f, \ccc n^{-1})_{p, w_n}$ is bounded above by a quantity larger than $c  E_{1}(f)_{p,w_{1}}$ which is worse than what \lem{lem5.2} implies.

Therefore, we do not really get anything useful that is valid for all doubling weights if $rp \leq 1$. In the case $rp > 1$, $\delta_0= \gamma_0 = 1$, and we are back to the same situation as in the case for $p\geq 1$, \ie
we can use the fact that any doubling weight is in $\W^{1,1}_1$. Hence, we get the following inverse theorem that is valid for all doubling weights.

 \begin{corollary} \label{conversethmplessonecor}
Let $w$ be a doubling weight,
$0<p<1$,  $f\in\Lp[-1,1]$, and let $r\in\N$ be such that $r > 1/p$.
 Then
\[ 
\w_\varphi^r (f, \ccc n^{-1})_{p, w_n}  \leq  {c \over n^{r-1/p}}   \left( \sum_{k=1}^{n}    k^{rp-2}     E_{k}(f)_{p,w_{k}}^p \right)^{1/p},
\]
where $\ccc$ is the constant from \cor{lem5.11cor}, and   the constant  $c$   depends  only on $r$, $p$ and the doubling constant of $w$.
\end{corollary}

\begin{remark}
Since any weight  that satisfies the $A^*$ property is in $\W^{0,0}$ (see Remark~\ref{remast}), it immediately follows from \thm{conversethmplessone} that, for $A^*$ weights $w$, we have
 \[
\w_\varphi^r (f, \ccc n^{-1})_{p, w_n}  \leq  {c \over n^{r}}   \left( \sum_{k=1}^{n}    k^{rp-1}     E_{k}(f)_{p,w_{k}}^p \right)^{1/p}, \quad 0< p <1 .
\]
 In fact, it is possible to show that one can set $\ccc = 1$ in this case.
\end{remark}

 \begin{proof}[Proof of \thm{conversethmplessone}] The method of the proof is rather standard (see \eg \cite{djl}). The beginning is the same as in the case $1\leq p \leq \infty$.
Namely,
let $P_n^* \in \Poly_n$ denote a polynomial of (near) best approximation to $f$ with weight $w_n$, \ie
\[
  \norm{f-P_n^*}{p,w_n} \leq  c E_{n}(f)_{p,w_n} .
\]
We let $N\in\N$ be such that $2^N \leq n < 2^{N+1}$, denote $m_j := 2^j$, and recall that $\ccc$ is the constant from \cor{lem5.11cor}.

Recalling that $w_n(x)\sim w_m(x)$ if $n \sim m$, and using \lem{lem5.2} we have
\begin{eqnarray*}
\w_\varphi^r (f, \ccc n^{-1})_{p, w_n}^p & \leq & \w_\varphi^r (f, \ccc 2^{-N})_{p, w_n}^p \\
& \leq & \w_\varphi^r (f - P_{m_N}^*, \ccc 2^{-N})_{p, w_n}^p + \w_\varphi^r (P_{m_N}^*, \ccc 2^{-N})_{p, w_n}^p \\
& \leq & c \norm{f - P_{m_N}^*}{p, w_{m_N}}^p + \w_\varphi^r (P_{m_N}^*, \ccc 2^{-N})_{p, w_{m_N}}^p \\
& \leq &
 c E_{m_N}(f)_{p,w_{m_N}}^p + \w_\varphi^r (P_{m_N}^*, \ccc 2^{-N})_{p, w_{m_N}}^p .
\end{eqnarray*}
Using
\[
P_{m_N}^* =  P_1^* + \sum_{j=0}^{N-1} (P_{m_{j+1}}^* - P_{m_{j}}^*)
\]
and \cor{lem5.11cor} with $t := 2^{-N}$ (noting that $t \leq 1/m_{j+1}$ for all $0\leq j \leq N-1$) we have
\begin{eqnarray*}
\w_\varphi^r (P_{m_N}^*, \ccc 2^{-N})_{p, w_{m_N}}^p   &\leq&
 \sum_{j=0}^{N-1} \w_\varphi^r  \left( P_{m_{j+1}}^* - P_{m_{j}}^*, \ccc 2^{-N}\right)_{p, w_{m_N}}^p \\
 & \leq &
 \sum_{j=0}^{N-1} \left( m_N \over m_{j+1} \right)^\delta (2^{-N} m_{j+1})^{rp} \norm{P_{m_{j+1}}^* - P_{m_{j}}^*}{p, w_{m_{j+1}}}^p \\
  & \leq &
 c \sum_{j=0}^{N-1} 2^{-(N-j)(rp-\delta)}  E_{m_j}(f)_{p, w_{m_j}}^p .
\end{eqnarray*}
Hence,
\[
\w_\varphi^r(f, \ccc n^{-1})_{p, w_n}^p \leq c   \sum_{j=0}^{N} 2^{-(N-j)(rp-\delta)}  E_{m_j}(f)_{p, w_{m_j}}^p ,
\]
and so as in the proof for $1\leq p\leq \infty$, we conclude that
\begin{eqnarray*}
\w_\varphi^r(f, \ccc n^{-1})_{p, w_n}^p & \leq & {c \over n^{rp-\delta}} \sum_{j=0}^{N} 2^{j(rp-\delta)}  E_{m_j}(f)_{p, w_{m_j}}^p \\
& \leq &
 {c \over n^{rp-\delta}} \left(  E_1(f)_{p, w_1}^p +     \sum_{j=1}^{N}   \sum_{k=m_{j-1}+1}^{m_j}  k^{rp-\delta-1} E_{k}(f)_{p, w_{k}}^p    \right)\\
&\leq &  {c \over n^{rp-\delta}}    \sum_{k=1}^{n} k^{rp-\delta-1}  E_{k}(f)_{p, w_{k}}^p .
\end{eqnarray*}
\end{proof}

We have the following immediate corollary of Theorems~\ref{jacksonthm} and \ref{conversethmplessone}.

\begin{corollary}[$0< p < 1$] \label{corppppp}
Let $r\in\N$,   $0 <p <1$ and $f\in\Lp[-1,1]$. Suppose that
 $w$ is a doubling weight from the class $\Wdg$ with $\gamma \leq rp$.    Then, for $0<\a< r-\delta/p$, we have
\[
E_n(f, [-1,1])_{p, w_n} = O(n^{-\a}) \iff  \w_\varphi^r(f, n^{-1})_{p, w_n} = O(n^{-\a}) .
\]
\end{corollary}

We remark that, since $1 \in \W^{0,0}$,   an immediate consequence of \cor{corppppp} is the usual equivalence result for unweighted polynomial approximation in $\Lp$ for $0<p<1$.

Again, taking into account that any doubling weight belongs to the class $\W^{1, 1}$  and assuming that $rp>1$ we get the following corollary.

\begin{corollary}[$0< p < 1$ and all doubling weights]
Let $w$ be a doubling weight, $0< p< 1$, $f\in\Lp[-1,1]$, and let $r\in\N$ be such that $r > 1/p$.   Then, for $0<\a< r-1/p$, we have
\[
E_n(f, [-1,1])_{p, w_n} = O(n^{-\a}) \iff  \w_\varphi^r(f, n^{-1})_{p, w_n} = O(n^{-\a}) .
\]
\end{corollary}

As a final remark in this section, we mention that it is still an open problem to prove or disprove if Theorems~\ref{conversethm} and \ref{conversethmplessone} are sharp.

\sect{$\K$-functionals and Realization} \label{kfn}

For $f\in\Lp$, $r\in\N$ and a weight $w$, the weighted $K$-functional is defined as follows
\[
\K_{r,\varphi}(f, t)_{p, w} := \inf_{g^{(r-1)} \in \AC_\loc} \left( \norm{f-g}{p, w} + t^r \norm{  \varphi^r g^{(r)}}{p, w} \right) ,
\]
where $\AC_\loc$ is the set of all locally absolutely continuous functions on $(-1,1)$.
In fact, for doubling weights $w$ we are interested in a sequence of these $K$-functionals with weights $w_n$,
and so we define several related quantities (all of which depend on $n$) as follows:
\[
\K_{r,\varphi_n}(f, t)_{p, w_n} := \inf_{g^{(r-1)} \in \AC_\loc} \left( \norm{f-g}{p, w_n} + t^r \norm{  \varphi_n^r g^{(r)}}{p, w_n} \right) ,
\]
where $\varphi_n(x) := \varphi(x) +1/n = n\dn(x)$,
\[
\Rea_{r,\varphi}(f, t)_{p, w_n} := \inf_{P_n \in \Poly_n} \left( \norm{f-P_n}{p, w_n} + t^r  \norm{  \varphi^r P_n^{(r)}}{p, w_n} \right) ,
\]
and
\[
\Rea_{r,\varphi_n}(f, t)_{p, w_n} := \inf_{P_n \in \Poly_n} \left( \norm{f-P_n}{p, w_n} + t^r  \norm{  \varphi_n^r P_n^{(r)}}{p, w_n} \right) .
\]
Note that $\Rea_{r,\varphi}$ and $\Rea_{r,\varphi_n}$ are   sometimes referred to as ``realizations'' of   appropriate $\K$-functionals or ``realization functionals'' (see \cites{dhi, d20}, for example).

 It is clear that
 \begin{eqnarray} \label{someeq}
 \K_{r,\varphi}(f, t)_{p, w_n} &\leq& \K_{r,\varphi_n}(f, t)_{p, w_n} \leq   \Rea_{r,\varphi_n}(f, t)_{p, w_n}     \andd \\ \nonumber
  \K_{r,\varphi}(f, t)_{p, w_n} &\leq& \Rea_{r,\varphi}(f, t)_{p, w_n} \leq \Rea_{r,\varphi_n}(f, t)_{p, w_n} , \quad t> 0.
 \end{eqnarray}

It follows from \thm{jacksonthm} that, if $w$ is a doubling weight, $r\in\N$, $0<p\leq \infty$, $f\in\Lp[-1,1]$, and $\A>0$ is any constant, then there exists $P_n\in\Poly_n$ such that
\[
\norm{f-P_n}{p, w_n} + n^{-r}  \norm{  \varphi_n^r P_n^{(r)}}{p, w_n}
  \leq c \tilde \w_\varphi^r(f, \A/n)_{p, w_n}    \leq c \w_\varphi^r(f, \A/n)_{p, w_n}            , \quad n\geq r,
\]
and hence, for any constant  $\B >0$,
\be \label{lowere}
\Rea_{r,\varphi_n}(f, t)_{p, w_n}
  \leq c \tilde \w_\varphi^r(f, \A/n)_{p, w_n}    \leq c \w_\varphi^r(f, \A/n)_{p, w_n}  , \quad  n\geq r \andd t \leq \B/n   ,
\ee
where the constant $c$ depends only on $r$, $p$, $\A$, $\B$,  and the doubling constant of $w$.

 Lemmas~\ref{lem5.2} and \ref{lem8.4} imply that,  if $w$ is a doubling weight, $1 \leq p\leq \infty$,  $f\in\Lp[-1,1]$, $\CC>0$, $\D>0$, and
$g$ is any function such that $ g^{(r-1)} \in \AC_\loc$ and $\norm{\varphi^r g^{(r)}}{p} < \infty$,   then
\[
\w_\varphi^r(f, t)_{p, w_n}   \leq   c \K_{r,\varphi}(f, \CC t)_{p, w_n} , \quad 0<t \leq \D/n ,
\]
where the constant $c$ depends only on $r$, $p$, $\CC$, $\D$,  and the doubling constant of $w$.

Therefore, together with \ineq{someeq}, this immediately implies the following result.

\begin{corollary}
If $w$ is a doubling weight, $1 \leq p\leq \infty$,  $f\in\Lp[-1,1]$, and $n,r\in\N$ are such that $n\geq r$,
and $\A/n \leq t \leq \B/n$,
 then
\begin{eqnarray*}
\w_\varphi^r(f, t)_{p, w_n}  &\sim&    \tilde \w_\varphi^r(f, t)_{p, w_n} \sim
\K_{r,\varphi}(f, t)_{p, w_n} \sim  \K_{r,\varphi_n}(f, t)_{p, w_n} \\
&\sim & \Rea_{r,\varphi}(f, t)_{p, w_n}  \sim  \Rea_{r,\varphi_n}(f, t)_{p, w_n} ,
\end{eqnarray*}
where all equivalence constants depend only on $r$, $p$, $\A$, $\B$,  and the doubling constant of $w$.
\end{corollary}

We now turn our attention to the case $0<p<1$. Things are a bit more complicated now since, as was shown in \cite{dhi}, the $\K$-functionals are identically zero if $0<p<1$.
However, we are still able to get the equivalence of the moduli and the realization functionals.

Lemmas~\ref{lem5.11new} and \ref{lem5.2} imply that, for a doubling weight $w$, $0<p<1$, $n\in\N$, and some constant $\ccc$ depending only on $r$, $p$, and the doubling constant of $w$,
\[
\w_\varphi^r (f, t)_{p, w_n} \leq c \Rea_{r,\varphi}(f, t)_{p, w_n}, \quad  0<t \leq \ccc/n ,
\]
where $c$ depends only on $r$, $p$ and the doubling constant of $w$.
For $n\geq r$, together with  \ineq{lowere}, this implies, for $\A/n \leq t \leq \ccc/n$,
 \begin{eqnarray} \label{trrr}
\Rea_{r,\varphi_n}(f, t)_{p, w_n} &\leq& c \tilde \w_\varphi^r (f, \A/n)_{p, w_n} \leq c \tilde \w_\varphi^r (f, t )_{p, w_n}
\leq c \w_\varphi^r (f, t)_{p, w_n} \\ \nonumber
&\leq&   c\Rea_{r,\varphi}(f, t)_{p, w_n} \leq
c \Rea_{r,\varphi_n}(f, t)_{p, w_n} .
\end{eqnarray}

Suppose now that $P_n^* \in \Poly_n$ is  a polynomial of (near) best approximation to $f$ with weight $w_n$, \ie
$ \norm{f-P_n^*}{p,w_n} \leq  c  E_{n}(f)_{p,w_n}$, and consider
\[
\Rea_{r,\varphi}^*(f, t)_{p, w_n} :=     \norm{f-P_n^*}{p, w_n} + t^r   \norm{  \varphi^r (P_n^*)^{(r)}}{p, w_n} .
\]
Then, Lemmas~\ref{lem5.11new} and \ref{lem5.2} and \thm{jacksonthm} imply, for $\A/n \leq t \leq \ccc/n$,
\begin{eqnarray*}
\Rea_{r,\varphi}^*(f, t)_{p, w_n} &\leq& c E_{n}(f)_{p,w_n} + t^r  \norm{  \varphi^r (P_n^*)^{(r)}}{p, w_n}
  \leq      c E_{n}(f)_{p,w_n} + c (\A/n)^r  \norm{  \varphi^r (P_n^*)^{(r)}}{p, w_n} \\
& \leq &  c E_{n}(f)_{p,w_n} + c \w_\varphi^r (P_n^*, \A/n)_{p, w_n}
 \leq  c E_{n}(f)_{p,w_n} + c \w_\varphi^r (f, \A/n)_{p, w_n} \\
 & \leq & c \w_\varphi^r (f, \A/n)_{p, w_n} \leq c \w_\varphi^r (f, t)_{p, w_n} .
\end{eqnarray*}
Since
\[
\Rea_{r,\varphi}(f, t)_{p, w_n}   \leq  \Rea_{r,\varphi}^*(f, t)_{p, w_n},
\]
together with \ineq{trrr} we get the following result.

\begin{corollary}
If $w$ is a doubling weight, $0< p <1$,  $f\in\Lp[-1,1]$, and $n,r\in\N$ are such that $n\geq r$, then there exists a positive constant $\ccc$ depending only on  $r$,  $p$ and  the doubling constant of $w$, such that,
for any constant $0<\A<\ccc$ and $\A/n \leq t \leq \ccc/n$, we have
\[
\w_\varphi^r (f, t)_{p, w_n} \sim  \tilde \w_\varphi^r (f, t)_{p, w_n} \sim     \Rea_{r,\varphi}(f, t)_{p, w_n}  \sim     \Rea_{r,\varphi_n}(f, t)_{p, w_n} \sim \Rea_{r,\varphi}^*(f, t)_{p, w_n} ,
\]
where all equivalence constants depend only on $r$, $p$, $\A$, $\ccc$,  and the doubling constant of the weight $w$.
\end{corollary}

\end{document}